\documentclass[12pt]{amsart}
\usepackage[margin=1.2in]{geometry}
\title[Recovering source terms from space-time samples]{Dynamical Sampling for the Recovery of Spatially Constant Source Terms in Dynamical Systems}

\usepackage{hyperref}
\author[Akram Aldroubi, Rocio Diaz Martin, Ivan Medri]{A. Aldroubi, R. Diaz Martin, I, Medri}

\date{\vspace{-5ex}}
\usepackage{amsthm, amsfonts}
\usepackage{ dsfont }
\usepackage{amssymb}
\usepackage{enumerate}
\usepackage{graphicx}
\usepackage{float}
\usepackage{bbm}
\usepackage{amsmath}
\usepackage{comment}
\usepackage{hyperref}
\usepackage{listings}
\usepackage{color}
\usepackage[dvipsnames]{xcolor}
\usepackage{enumitem}
\usepackage{subcaption}

\allowdisplaybreaks

\hypersetup{
    colorlinks,
    citecolor=blue,
    filecolor=blue,
    linkcolor=blue,
    urlcolor=blue
}

\newtheorem{theorem}{Theorem}[section]
\newtheorem{proposition}[theorem]{Proposition}
\newtheorem{lemma}[theorem]{Lemma}

\newtheorem{corollary}[theorem]{Corollary}
\newtheorem{definition}[theorem]{Definition}
\newtheorem{remark}[theorem]{Remark}

\newtheorem{notation}[theorem]{Notation}

\newtheorem{example}{Example}

\newtheorem{problem}{Problem}

\newcommand{\R}{\mathbb{R}}

\newcommand{\N}{\mathbb{N}}

\newcommand{\bigzero}{\mbox{\normalfont\Large\bfseries 0}}

\newcommand{\rank}{{\rm rank\,}}

\usepackage[makeroom]{cancel}

\def\N{\mathbb{N}}

\def\R{\mathbb{R}}
\def\C{\mathbb{C}}
\def\F{\mathbb{F}}

\def\I{\mathbb{I}}

\def\ker{\mathop{Ker}}
\usepackage{url}

\newcommand{\ak}[1]{\textcolor{aacolor}{#1}}
\definecolor{aacolor}{rgb}{0.05, 0.75, 1}

\begin{document}
	\date{}
\address{\textrm{(Akram Aldroubi)}
Department of Mathematics,
Vanderbilt University,
Nashville, Tennessee 37240-0001 USA}
\email{aldroubi@math.vanderbilt.edu}

\address{\textrm{(Roc\'io Mart\'in D\'iaz)}
	Department of Mathematics,
	Vanderbilt University,
	Nashville, Tennessee 37240-0001 USA}
\email{rocio.p.diaz.martin@vanderbilt.edu}

\address{\textrm{(Ivan Medri)}
Department of Mathematics,
	Vanderbilt University,
	Nashville, Tennessee 37240-0001 USA}
\email{cocarojasjorgeluis@gmail.com}

\thanks{
Akram Aldroubi is supported in part by NSF Grant NSF/DMS-2208030.
}


\maketitle
\begin{abstract}
In this paper, we investigate the problem of source recovery in a dynamical system utilizing space-time samples. This is a specific issue within the broader field of dynamical sampling, which involves collecting samples from solutions to a differential equation across both space and time with the aim of recovering critical data, such as initial values, the sources, the driving operator, or other relevant details. Our focus in this study is the recovery of unknown, stationary sources across both space and time, leveraging space-time samples. This research may have significant applications; for instance, it could provide a model for strategically placing devices to measure the quantity of pollutants emanating from factory smokestacks and dispersing across a specific area.
Space-time samples could be collected using measuring devices placed at various spatial locations and activated at different times. We present necessary and sufficient conditions for the positioning of these measuring devices to successfully resolve this dynamical sampling problem. This paper provides both a theoretical foundation for the recovery of sources in dynamical systems and potential practical applications. 	
	
\end{abstract}

\section {Introduction}    
   In various applications, such as environmental monitoring, it is crucial to accurately determine the location of pollution sources and assess the magnitude of their emissions. Typically, this information is obtained by deploying sensors across different locations. 
    Consequently, these scenarios can be framed as sampling and reconstruction problems.
    For instance, in the case of atmospheric emissions, considered in \cite{RDCV12} and the references therein, the goal is to recover the pollution intensity emitted by $K$ smokestacks by using $L$ sensors.

    In the present work, we examine the case where source emissions are constant in time, and atmospheric pollution evolves through a linear discrete-time dynamic. To illustrate, let us consider the following dynamical system:
    \begin{equation}\label{discreteGen}
        x({n+1}) = A^*x(n)+\underbrace{\sum\limits_{i=1}^K c_i \, e_{j_i}}_{\text{source term}},\qquad x(0)=x_0.
    \end{equation}
    Here, $x(n)$ represents the state of contamination at discrete time steps indexed by $n\in \N_0:=\N\cup{0}$. We assume that each state $x(n)$ is a vector in a finite-dimensional vector space, such as $\R^d$. The system initializes at an unknown vector $x_0$, and then the dynamic is governed by the linear operator $A^*$ and by the source term. For each $i=1,\dots,K$, the canonical vector $e_{j_i}$ represents the location of the $i$-th factory in an industrial zone, which releases plumes at a constant rate $c_i$.

    Regardless of the initial state $x_0$, the main objective is to estimate the scalar values $c_1,\dots,c_K$ using space-time samples of the form:
    \begin{equation}\label{eq: canonic samples}
    \langle x(n), e_{k_\ell}\rangle
    \end{equation}
    These samples are collected at specific locations $k_\ell$, $\ell=1,\dots,L$, where $\langle\cdot ,\cdot \rangle$ denotes the canonical inner product in $\R^d$. The new set of $L$ canonical vectors $e_{k_1},\dots,e_{k_L}$ can be viewed as the sensors in this context. (See Fig. \ref{fig: city})

    The problem above is one of many problems that fall under the umbrella of dynamical sampling problems. 
    For example, when all $c_i$ in system \eqref{discreteGen} are set to zero (i.e., $c_i=0$ for $i=1,\dots,K$), and assuming the operator $A^*$ is known, the objective is to determine the necessary and sufficient conditions on the spatial sampling vectors for recovering the initial condition $x_0$. This particular task is referred to as the space-time trade-off sampling problem (see, for e.g., \cite{APT15,ACCMP17, ACMT17,  ADK13, ADK15, AGHJKR21,  AP17,CMPP20, CJS15, DMM21, MMM21, UZ21, ZLL17}). If, on the other hand, the operator $A^*$ is unknown, it leads to the system identification problem presented in \cite{AHKLLV18, AK16, Tan17}. Lastly, in \cite{AGK23, AHKK23, AK16}, the focus is on identifying specific types of source terms that drive the dynamical system.  Dynamical sampling problems are connected to several areas of mathematics, including frame theory, control theory, functional analysis, and harmonic analysis \cite{AKh17, ACCP21,  BH23, BK23,CMPP20, CH19, CH23,  DMM21, FS19, GRUV15, KS19, MMM21, Men22, MT23, ZLL17}. These problems also have numerous applications in science and engineering \cite{ RBD21,AD20, MBD15, MD17}.
    
    Let us return to our sampling and source reconstruction problem in order to state it in a more general manner. Instead of using the measurements \eqref{eq: canonic samples} we can consider more general samples in the direction of vectors $b_\ell\in\R^d$, that is, 
    \begin{equation}\label{eq: measurements}
        y_\ell(n):=\langle x(n), b_{\ell}\rangle, \qquad \text{ for } \ell=1,\dots,L.
    \end{equation}
    Moreover, it is worth noting that we can treat the source term \eqref {discreteGen} as a vector in the subspace  
    \begin{equation*}\label{eq: source space}
        W:=span\{e_{j_1},\dots,e_{j_K}\} .
    \end{equation*}    Then, by expressing any vector in $W$ as a linear combination of $e_{j_1},\dots,e_{j_K}$, that is, 
    \begin{equation}\label{eq: source term}   
        \omega=\sum\limits_{i=1}^K c_i \, e_{j_i},
    \end{equation}
    we can observe that the problem of recovering $c_1,\dots,c_K$ is equivalent to that of recovering the vector $\omega$ in the subspace $W$. 
    
    Precisely, we state our general problem as follows.     
    \begin{problem}\label{prob: general}
        Let $\F$ be either $\R$ or $\C$, $A:\F^{d}\to\F^{d}$ be a linear operator with adjoint $A^*$, and  $W$ be  a subspace of $\F^d$. Given vectors $b_1,\dots, b_L\in\F^d$, consider the discrete-time dynamic 
        \begin{equation}\label{eq: dyn system subject to}
            \begin{cases}
            x({n+1}) = A^*x(n)+\omega,  \qquad 
            \text{subject to } \omega \in W\\
            y_\ell(n)=\langle x(n),b_\ell\rangle, \qquad \text{ for } \ell\in\{1,\dots, L\}, 
            \end{cases}
        \end{equation} 
        where the initial state $x(0)=x_0$ is unknown. 
        The \textit{source recovery problem} consists of providing necessary and sufficient conditions on such vectors $b_1,\dots,b_L$ so that the constant source $\omega$ can be recovered from a finite number of linear operations using a finite number of measurements $y_\ell(n)$. 
         
         For convenience and brevity, we adopt the following definitions.
    \end{problem}
     \begin {definition} ${}$\newline \label {OV}
     \begin {enumerate}
     \item The  set of vectors $\{b_1,\dots,b_L\}$ used in Problem \ref {prob: general} will be called \textbf{spatial observational vectors.}
     \item If the set of spatial observational vectors $\{b_1,\dots,b_L\}$ allows the reconstruction of  any $\omega \in W$, we say that the spatial observational vectors are {\textbf{complete}}  for this source recovery problem.
     \end {enumerate} 
        \end{definition}
  
{ The spatial observational vectors represent physical measuring devices used to acquire space-time samples. Specifically, if $b_j=e_{l_j}$, where $e_{l_j}$ is a standard basis element at spatial position $l_j$ for $\F^d$, then the set $\{b_1,\dots,b_L\}$ corresponds to measuring devices positioned at $l_j$, $1\le j\le l$. These devices can be activated at various time instances $n$ to collect space-time samples. }
      
    We will strive to find answers to the following questions:
    \begin{itemize}
        \item If the subspace $W$ and the sampling vectors $b_1,\dots,b_L$ are given in advance, is it possible to recover any $\omega \in W$?
        \item Given the subspace $W$, is there a minimal set $b_1,\dots,b_L$ such that reconstructing any $\omega \in W$ is possible?
        \item In case that the source reconstruction is possible, how many time iterations and spatial measurements are needed?
    \end{itemize}

\begin{figure}
     %\centering
     \begin{subfigure}[b]{0.45\textwidth}
         \centering
         \includegraphics[width=\textwidth]{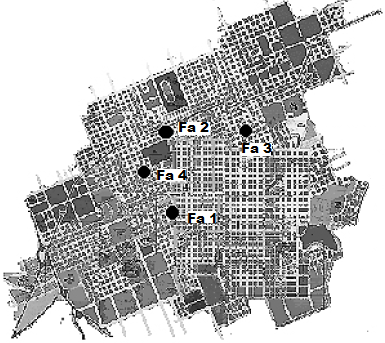}
         %\caption{Image credits \cite{chaparro}. The city of Tandil in Buenos Aires Province, Argentina is depicted on a map, featuring four factories labeled Fa.1 to Fa.4. Additionally, two natural zones, Na.1 and Na.2, are shown. In their study, the authors investigated the magnetic properties of the soil's top layers, exploring the hypothesis of a correlation with the atmospheric fallout of pollutants generated by the metallurgical factories. }
         %\label{fig: real city}
     \end{subfigure}
     \hfill
     \begin{subfigure}[b]{0.48\textwidth}
        \centering
         \includegraphics[width=\textwidth]{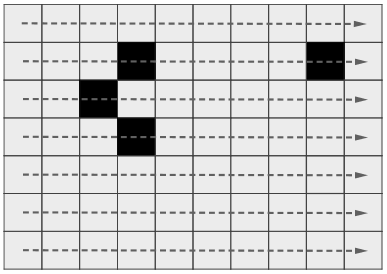}
     \end{subfigure}
%\label{fig: city}
\caption{Left Panel: Image credits \cite{chaparro}. The city of Tandil in Buenos Aires Province, Argentina is depicted on a map, featuring four factories labeled Fa 1 to Fa 4. In the article \cite{chaparro}, the authors investigated the magnetic properties of the soil's top layers, exploring the hypothesis of a correlation with the atmospheric fallout of pollutants generated by metallurgical factories. 
This pollution study can be modeled in our scenario by viewing the wind and precipitation evolution as the dynamical system, and the measurements taken in the surface layer as the spatial samples. Right panel: As a simple visualization, the city on the left panel is represented on the right panel as a $7\times 10$ grid where the pollution sources (factories) are shown as black pixels and the wind evolution is depicted as dotted vectors from right to left. Therefore, the city can be viewed as $\R^{70}$. Let $e_1\in\R^{70}$ represent the top left location in the grid, and then let draw the rest of the canonical vectors in $\R^{70}$ indicating locations from left to right and from top to bottom (so, for e.g., $e_{10}\in\R^{70}$ represents the top left location, $e_{61}$ the bottom left, and $e_{70}$ the top bottom right). Then,  
the pollution sources can be symbolized by the canonical vectors $e_{14}$, $e_{19}$, $e_{23}$, $e_{34}$. Finally, the wind and rain can be modelled as a nilpotent operator given by a $70\times 70$ matrix with $7$ blocks of size $10\times 10$.}
\label{fig: city}
\end{figure}

\subsection{Notation}

    Throughout this paper, $\F$ will be either the field of real numbers $\R$ or the field of complex numbers $\C$. Then, $\F^d$ is the natural $d$-dimensional vector space over $\F$, and $\F[x]$ is the ring of all polynomials with coefficients in $\F$. As usual, the canonical basis for $\F^d$ will be denoted by $\{e_1,\dots,e_d\}$ and, if needed, a vector $v\in\F^d$ is written in canonical coordinates reads as $v=(v^1,\dots,v^d)^T$. For a subspace $V$ of $\F^d$, its dimension will be $\mathrm{dim}(V)$, and $V^\perp$ will be its orthogonal complement. For a linear operator $A$, we denote by $A^*$ its adjoint operator and by $\sigma(A)$ its spectrum.  Also, $I$ stands for the identity matrix. We will write ${\bf 0}$ for the zero vector or for matrices with all entries equal to zero.  The set $\N_0$ denotes the set of non-negative integers, defined as $\N_0 := \N \cup \{0\}$.
   
   A recurring notation used throughout this paper is that of the minimal $A$-annihilating polynomial of $b$, defined as follows:
\begin{definition} \label{MAPB}
Given a linear operator $A:\F^d\to\F^d$.  The minimal monic polynomial that annihilates $A$ is denoted by $m_A$. Given a vector $b \in \F^d$, the minimal $A$-annihilating polynomial of $b$ is defined to be the unique monic polynomial $m_b \in \F[x]$ of least degree such that $m_b(A)b = \textbf{0}$ (see, for example, \cite{Hoffman}). Note that $m_b$ divides $m_A$.
\end{definition}

\subsection{Contributions and Organization}  

%In this paper, the problem of recovering recovering the source term of a non-homogeneous dynamic. Our main results are presented in Section \ref{sec: main results} while proofs are collected in Section \ref{sec: proofs}. 

Section \ref{sec: main results} contains our main results. We first define the observational augmented orbits (see \eqref{eq: span of measurements}) that are fundamental for studying Problem \ref {prob: general}. Theorem \ref {th: prop_characterization_by_eq} characterizes the observational augmented orbits. In Subsection \ref{subsec: nec suf cond}, we determine necessary and sufficient conditions for the spatial observational vectors $b_1,\dots,b_L$ to be complete (see Definition \ref {OV} above). Our first result in this regard is presented in Theorem \ref{coro: nec suff cond}. This theorem enables us to derive more concrete necessary and sufficient conditions. For instance, when $1$ is not an eigenvalue of $A$, we present Theorem \ref{coro: nec and sufi no eigenval 1}, which offers easily verifiable necessary and sufficient conditions for the completeness of observational vectors $b_1,\dots, b_L$ (Sub-subsection \ref{subsubsec: not 1}). We then provide a general test for verifying the completeness of the observational vectors $b_1,\dots, b_L$ (Theorem \ref{th: general test} in Sub-subsection \ref{subsub: general}). We also establish that the number of observational vectors $L$ required for completeness must be at least equal to the dimension $\dim W$ of the source space $W$ (see Corollaries \ref{coro: at least K} and \ref{coro: at least K general}). This raises the question of when it is possible to reconstruct the source term using \textit{exactly} the same number of spatial observational vectors as the dimension of $W$. In Subsection \ref{Single Source Term}, we investigate this question when $\dim W=1$. Throughout this article, we include illustrative examples to enhance understanding, motivate certain ideas, and clarify specific issues. The proofs of the results are relegated to Section \ref{sec: proofs}.

 {   \section {main results}\label{sec: main results}
In this section, we aim to establish necessary and/or sufficient conditions for a set of spatial observational vectors
$\{b_1,\dots,b_L\}$ for the dynamic \eqref{eq: dyn system subject to} to be complete, as stated in Theorems \ref{coro: nec suff cond}, \ref {coro: nec and sufi no eigenval 1} and their corollaries. Before we proceed, we begin by rewriting the states and measurements of the dynamic in Problem \ref{prob: general}.
}

 {Denoting the initial state by $x(0)=x_0$, the evolution of the states ($x(n)=A^*x(n-1)+\omega$ in \ref{prob: general}) has the form
    \begin{equation}\label{eq: dyn evol}
        x(n) = (A^*)^n x_0 +\sum_{j=0}^{n-1}(A^*)^j\omega, \qquad \forall n\in \N.
    \end{equation}
}
{Using \eqref {eq: dyn evol}, the time-space measurements in \eqref{eq: dyn system subject to} can be written in the following form:
\begin{align}\label{data}
        y_{\ell}(n)&= \langle x_0, A^n b_{\ell}\rangle +\langle \omega, \Lambda_n b_{\ell}\rangle\\
        &=\left[{\begin{array} {rrr}
        A^nb_\ell \\
        \Lambda_nb_\ell
        \end{array} }\right]^*
        \left[{\begin{array} {rrr}
        x_0\\
        \omega
        \end{array} }\right] \qquad  \nonumber
    \end{align}
    for $n\in\N_0$ and $\ell\in {1,\dots,L}$, where  
    \begin{equation}\label{Lambda_powers of A}
        \Lambda_n:=\sum_{j=0}^{n-1}A^j \qquad \text{ for } n\in \N, \qquad  \Lambda_0:={\bf 0} \text{ (the zero-matrix).}
    \end{equation}
 }

{  The last equality in \eqref {data}, is written in matrix form, where a $d\times 2d$ matrix is multiplied by a vector in $\F^{2d}$ with the first $d$ components representing $x_0$ and the last $d$ components representing $\omega$.}

 {Equation \eqref {data} leads us to introduce two fundamental spaces associated with    $\{b_1,\dots,b_L\}$. }

{  The first fundamental space is the  \textit {observational orbits  space} $Z(A;b_1,\dots,b_L)$ defined by
\begin{equation}\label{eq: Z}
            Z(A;b_1,\dots,b_L):=span\{A^nb_\ell:\, n\in\N_0, \, \, \ell\in\{1,\dots,L\}\}.
  \end{equation}
  }

{ The second space we introduce is  referred to as the \textit{augmented observational orbits  space} defined by
 \begin{align}\label{eq: span of measurements}
        \mathcal{O}(b_1,\dots,b_\ell):=span &\left\{\left[ {\begin{array} {rrr}
        A^nb_\ell \\
        \Lambda_nb_\ell
        \end{array} }\right]\in\F^{2d}: \, n\in\N_0, \, \ell=1,\dots,L\right\},  
    \end{align}   
 where $\Lambda_n$ is as in \eqref {Lambda_powers of A}. 
  }

 The space $\mathcal{O}(b_1,\dots,b_L)$ is a subspace of $\F^{2d}$. Its dimension plays a crucial role in determining necessary and sufficient conditions in Theorem \ref{coro: nec suff cond} for the completeness of the spatial observational vectors ${b_1,\dots,b_L}$ (see Definition \ref {MAPB}). For the case of $L=1$, it can be shown that $\dim \mathcal{O}(b_1)\le d+1$, and that the dimension $d$ can be achieved by  an appropriately chosen vectors $b_1$. Intuitively, one might expect that for any $L\ge 2$, the dimension of $\mathcal{O}(b_1,\dots,b_L)$ could be as large as $2d$. However, the following theorem states that $\dim \mathcal{O}(b_1,\dots,b_L)\le \min \{d+L, 2d\}$.
 
   \begin{theorem}\label{th: prop_characterization_by_eq} 
        Let $A:\F^{d}\to\F^d$ and vectors $b_1,\dots,b_L\in \F^d\setminus\{{\bf 0}\}$.
        Then, the subspace $\mathcal{O}(b_1,\dots,b_L)\subset \F^{2d}$ of space-time sampling vectors defined in \eqref{eq: span of measurements} 
        coincides with  the set of vectors
        \begin{align}\label{eq: characterizing space}
            \left\{
            \left[ {\begin{array} {r}
            (A-I)g+\displaystyle\sum_{\ell=1}^L \mu_\ell\, b_\ell \\g\qquad \qquad \qquad \end{array} }\right]\in\F^{2d}: \, g\in Z(A;b_1,\dots,b_L), \, \mu_1,\dots,\mu_L\in\F \right\}.
            \end{align}
            In particular, $\dim \mathcal{O}(b_1,\dots,b_L)\le \min \{d+L, 2d\}$. 
    \end{theorem}
  
  \begin {remark} \label {KS} {When there is a single observational vector $b_1$ (i.e., $L=1$), the space of orbits $Z(A;b_1)$ is referred to as the Krylov subspace of the highest order generated by $A$ and $b_1$. It is well-known that the dimension of $Z(A;b_1)$ is equal to the degree $r$ of the minimal $A$-annihilating polynomial $m_{b_1}(x)$ of $b_1$. In fact, if $r$ denotes the degree of $m_{b_1}$, then ${b_1,Ab_1,\dots, A^{r-1}b_1}$ forms a basis for $Z(A;b_1)$.}
\end{remark}

  {
 \begin{corollary}\label{rank_d+1} 
        Let $d\ge 2$  and $A:\F^{d}\to\F^d$ be a linear operator. Consider a non-zero vector $b\in \F^d$. Then, the dimension of the space $\mathcal{O}(b)$ is equal to $r+1$, where $r$ is the degree of the minimal $A$-annihilating polynomial of $b$.
    \end{corollary} }
    
   \subsection {Necessary and Sufficient conditions for completeness of observational vectors} \label{subsec: nec suf cond}
   
   {Our first result provides the necessary and sufficient conditions for the set ${b_1,\dots,b_L}$ to be a complete observational set for Problem  \ref{prob: general}. However, its primary utility lies in deriving additional more useful necessary and/or sufficient conditions.}
   
    \begin{theorem}\label{coro: nec suff cond}
        Let $A:\F^d\to\F^d$ be a linear operator, and let $W$ be a subspace of $\F^d$ with $P_W:\F^d\to\F^d$ the orthogonal projection onto $W$. A set of spatial  {observational vectors}  $\{b_1,\dots,b_L\}\subset \F^d$ {is complete} for Problem \ref{prob: general} (see Definition \ref {OV}) if and only if there exist vectors $g_1,\dots,g_N\in Z(A;b_1,\dots, b_L)$, for some $N\in\N$, such that 
        \begin{enumerate}
            \item  \label{item: 2nd cond} 
            there exists a solution $M\in\F^{L\times N}$ for the matrix problem
            \begin{equation}\label{eq: nec cond 2}
                B M=(A-I) G
            \end{equation}
            where $B$ is the $d\times L$ matrix having $b_1,\dots,b_L$ as columns, and $G$ is the $d\times N$ matrix with columns $g_1,\dots,g_N$, and 
            \item \label{item: 1st cond} $\{P_W(g_1),\dots,P_W(g_N)\}$ spans $W$. 
        \end{enumerate}
    \end{theorem}

\subsubsection{Completeness of observational vectors when $1\notin \sigma(A)$.}\label{subsubsec: not 1}

We now use Theorem \ref{coro: nec suff cond} to derive simple, necessary and sufficient conditions, enabling us to directly determine whether a set of observational vectors is complete.

     \begin{theorem}\label{coro: nec and sufi no eigenval 1}        Let $A:\F^{d}\to\F^d$ be a linear operator with spectrum $\sigma(A)$, and let $W$ be a $K$-dimensional subspace of $\F^d$. Furthermore, assume that $1\notin \sigma(A)$,  then, the set of spatial observational vectors $\{b_1,\dots,b_L\}\subset \F^d\setminus\{{\bf 0}\}$ is complete for Problem \ref{prob: general} if and only if
        \begin{equation}\label{eq: suf cond no eig 1}
            rank \left(P_W(A-I)^{-1} B\right) = K
        \end{equation}
        where $P_W$ denotes the orthogonal projection onto $W$, and $B$ is the matrix having $b_1,\dots,b_L$ as column vectors.
    \end{theorem}

As a corollary, we conclude that the set of observational vectors must contain at least $K$ vectors when $\dim W=K$. 
   \begin{corollary}\label{coro: at least K}
        Let $A:\F^d\to\F^d$ be a linear operator with $1\notin \sigma(A)$, and let $W$ be a $K$-dimensional subspace of $\F^d$. If the set $\{b_1,\dots,b_L\}$ forms a complete observational set for Problem \ref{prob: general}, then $L\ge K$.
    \end{corollary}
    
It turns out that, when $1\notin \sigma (A)$, a complete set of observational vectors $\{b_1,\dots,b_K\}$ can always be found, where $K=\dim W$.
\begin{corollary}
\label{coro: at least as many as dim W}
    Let $A:\F^d\to\F^d$ be a linear operator, and  let $W$ be a subspace of $\F^d$. If $\dim W=K$  and $1 \notin \sigma (A)$, then there exists a complete set of spatial observational vectors $\{b_1,\dots,b_K\}$. In particular, if ${\omega_1,\dots,\omega_K}$ form a basis for $W$,  a possible choice is given by $b_k=(I-A)\omega_{k}$ for $1\leq k\leq K$. 
\end{corollary}

\begin{remark}${ }$ 
 \begin{enumerate}
 \item The significance of Condition \eqref{eq: suf cond no eig 1} in Theorem \ref{coro: nec and sufi no eigenval 1} is that it provides a practical method to verify whether a given set of observational vectors is complete.
\item When the source aligns with \eqref{eq: source term}, the $P_W$ in \eqref{eq: suf cond no eig 1} simplifies to a diagonal matrix with diagonal entries from the set $\{0,1\}$.
\item The selection of the spatial observational vector in Corollary \ref{coro: at least as many as dim W} represents merely one potential choice for a complete set of observational vectors. However, when each vector $b_k$ is required to be a standard basis element of $\F^d$, we must choose a set of vectors ${e_{l_1},\dots,e_{l_K}}$ and validate their completeness using Theorem \ref{coro: nec and sufi no eigenval 1}.
 \end{enumerate}
 \end{remark}

  {Theorem \ref{coro: nec and sufi no eigenval 1} above provides a 
  simple test for determining whether a set of spatial observational vectors $\{b_1,\dots,b_L\}\subset \F^d\setminus\{{\bf 0}\}$ is complete for Problem \ref{prob: general} under the assumption that $1 \notin \sigma(A)$. Later, in  Theorem \ref{th: general test} we will provide a general test without the assumption that $1\not\in \sigma(A)$.  Allowing $1$ to be an eigenvalue of $A$  is more difficult to analyze.  In particular, if $1\in \sigma(A)$, then Corollary \ref{coro: at least as many as dim W} is false as the following example shows.} 
 \begin{example}\label{example_no_sol} In dimension $2$, consider the dynamical system 
\begin {equation} \label {dynamic for ex}
x({n+1}) = A^*x(n)+c_1 \, \omega_1  
\end{equation}  
where 
$$A:=\left[ {\begin{array} {rr}
 1  & 0 \\
 1 & 1  
 \end{array} }\right] \qquad \text{ and } \qquad \omega_1:=\left[ {\begin{array} {r}
 1 \\
 0  
 \end{array} }\right]. $$ 
The goal is to recover the unknown source intensity $c_1\in\F$. 

The eigenspace of $A$ corresponding to eigenvalue $1$ is spanned by the vector
$$\left[ {\begin{array} {r}
 0 \\
 1  
 \end{array} }\right]$$
and is orthogonal to $W=span \{\omega_1\}$. Suppose that there exists a vector 
$$b=\left[ {\begin{array} {r}
 b^1 \\
 b^2  
 \end{array} }\right]$$
such that, by sampling the states $x(n)$, we can recover $c_1$ in \eqref{dynamic for ex}. By denoting the initial state as 
$$x_0=\left[ {\begin{array} {r}
 x_0^1 \\
 x_0^2  
 \end{array} }\right],$$ 
 the measurements will be of the form:
\begin{equation*}
    \begin{cases}
    y(0) = x_0^1 \, b^1+ x_0^2 \, b^2\\
    y(1) = x_0^1\, b^1 + x_0^2 \, (b^2+b^1) + b^1 \, c_1\\
    y(2)=x_0^1\, b^1 + x_0^2 \, (b^2+2b^1) + 2b^1 \, c_1\\
    \vdots\\
    y(n)  = x_0^1\, b^1 + x_0^2 \, (b^2+nb^1) + nb^1 \, c_1\\
    \vdots
    \end{cases}
\end{equation*}
The equation above shows that to recover $c_1$ from the observations $y(j)$, then  necessarily $b^1\ne 0$. 
However, for any $c_1\in\F$, if $$x_0=c_1
\left[ {\begin{array} {r}
 b^2/b^1 \\
 -1  \, \, \, 
 \end{array} }\right]
$$ 
then $y(n)=0$ for all $n\in \N_0$. Therefore, for any choice of $b$, there exists $x_0$ such that $y_n=0$ for every $c_1$, and hence $c_1$ cannot be determined. Hence, if $1 \in \sigma (A)$ and $\mathrm{dim}(W) = 1$, having only one spatial sampling vector $b_1:=b$ might not be enough for solving Problem \ref{prob: general}.
\end{example}   

\subsubsection {Completeness of observational vectors: General case.}\label{subsub: general}

In order to state the general test for determining whether or not a set of spatial observational vectors is complete for Problem \ref{prob: general}, we need to introduce some definitions and notation. This general result is also based on Theorem \ref{coro: nec suff cond}.

\begin{notation} \label {HatNotation}
    Given $p\in\F[x]$, we denote by $\widehat{p}$ the polynomial
    \begin{equation}\label{eq: notation pol}
      \widehat{p}(x):=\frac{p(x)-p(1)}{x-1}  \in \F[x].
    \end{equation}
\end{notation}

\begin{definition} \label{AC}
    (See \cite{Hoffman}.)
    Given $A:\F^d\to\F^d$ a linear operator, a vector $b\in\F^d$, and an $A$-invariant subspace $V\subseteq \F^d$, a polynomial $\kappa\in\F[x]$ is said to bethe \textit{minimal $A$-conductor of $b$ into $V$} if it is the (unique) monic generator of the polynomial ideal
    $\mathbb{I}:=\{p\in\F[x]: \, p(A)b\in V \}.$
\end{definition}

Given $A$ a linear operator in $\F^d$ and  non-trivial vectors $b_1,\dots,b_L\in\F^d$, in what follows we will consider, for $j=1,\dots L,$ the minimal $A$-conductor polynomial $\kappa_j\in \F[x]$ of $b_j$ into the observational orbits  space 
$$V_j:=Z(A;b_1,\dots,b_{j-1})=\sum_{i=1}^{j-1}Z(A;b_i), $$ where $V_1:=\{\textbf{0}\}$. That is, $\kappa_j$ is the  generator of the following principal polynomial ideal 
    $$\mathbb{I}_j:=\{p\in\F[x]:\, p(A)b_j\in Z(A;b_1)+\dots+Z(A;b_{j-1})\},$$
    with the understanding that $\mathbb{I}_1:=\{p\in\F[x]: \, p(A)b_1=\textbf{0}\}$.
    For example, note that $\kappa_1$ is  the minimal $A$-annihilating polynomial $m_{b_1}$ of $b_1$. Moreover, if $b_2\in Z(A;b_1)$, then $\kappa_2(x)=1$ for all $x$. If instead $Z(A;b_1)\cap Z(A;b_2)=\{\textbf{0}\}$, then $\kappa_2=m_{b_2}$.

Theorem \ref{th: general test}, below,  extends Theorem \ref{coro: nec and sufi no eigenval 1} by eliminating the assumption that $1\notin \sigma(A)$. It provides a method for determining whether a set of spatial observational vectors ${b_1,\dots,b_L}\subset \F^d\setminus{{\bf 0}}$ is complete for Problem \ref{prob: general}. Before we delve into Theorem \ref{th: general test}, we need to introduce some constructs that are essential to its understanding.

Let $\{b_1,\dots,b_L\}\subset \F^d$ be a set of non-trivial spatial observational vectors.
    For each $1\leq j\leq L$, consider  the minimal $A$-conductor polynomial $\kappa_j\in \F[x]$ of $b_j$ into the observational orbits space  $V_j:=\sum_{i=1}^{j-1}Z(A;b_i)$ (where $V_1:=\{\textbf{0}\}$, and $\kappa_1:=m_{b_1}$). 
    For each $2\leq j\leq L$, let $q_j^i\in\F[x]$, for $1\leq i\leq j-1$, be such that 
    \begin{equation}\label{eq: q polynomials}
     \kappa_j(A)b_j=q_j^1(A)b_1+\dots+q_j^{j-1}(A)b_{j-1},  
    \end{equation}
    and define the vectors
    \begin{equation}\label{eq: characteristic vectors general case}
       g_j:=\widehat{\kappa_j}(A)b_j-\sum_{i=1}^{j-1}\widehat{q_j^i}(A)b_i, \qquad \text{ for } 1\leq j\leq L,
    \end{equation}
 where $\widehat{\kappa_j}$ and $\widehat{q_j^i}$ are defined using Notation \ref {HatNotation}. Using the constructed vectors above we now state  the theorem. 
 \begin{theorem}\label{th: general test}
    Let $A:\F^{d}\to\F^d$ be a linear operator, let $W$ be a subspace of $\F^d$ with $P_W:\F^d\to\F^d$ the orthogonal projection onto $W$, and let $\{b_1,\dots,b_L\}\subset \F^d$ be a set of non-trivial spatial observational vectors.
   Then, the set of spatial observational vectors $\{b_1,\dots,b_L\}$ is complete for Problem \ref{prob: general} if and only if $L\geq \mathrm{dim}(W)$ and
   \begin{equation} \label {GenCondComp}
      span\left\{ P_W(g_j): \, 1\leq j\leq L\right\}=W,
   \end{equation}
   where $g_j$ are as in \eqref {eq: characteristic vectors general case}.
 \end{theorem}
 \begin{remark}
    Condition \eqref{eq: suf cond no eig 1} in Theorem \ref{coro: nec and sufi no eigenval 1} can be rewritten as
    \begin{equation*}
        span\{P_W(A-I)^{-1}b_j:\, 1\leq j\leq L\}=W
    \end{equation*}
    and it coincides with Condition \eqref {GenCondComp} in Theorem \ref{th: general test} for the case $1\not\in \sigma(A)$.  
    Indeed, if  $1\not\in\sigma(A)$, by using \eqref{eq: q polynomials}, we can rewrite the vectors defined by \eqref{eq: characteristic vectors general case}
    as
    \begin{align}
       g_j&=\widehat{\kappa_j}(A)b_j-\sum_{i=1}^{j-1}\widehat{q_j^i}(A)b_i\notag\\
       &=(A-I)^{-1}\left((\kappa_j(A)-I)b_j-\sum_{i=1}^{j-1}(q_j(A)-I)b_i \right)\notag\\
       &=(A-I)^{-1}\left(\left(\kappa_j(A)b_j-\sum_{i=1}^{j-1}q_j(A)b_i\right)-b_j+\sum_{i=1}^{j-1}b_i \right)\notag\\
       &=(A-I)^{-1}(-b_j+\sum_{i=1}^{j-1}b_i). \label{eq: charact vectors 1 not in spec}
    \end{align}
    Since
    \begin{equation*}
        span\{b_j:\, 1\leq j\leq L\}=span\{-b_j+\sum_{i=1}^{j-1}b_i:\, 1\leq j\leq L\},
    \end{equation*}
    by using \eqref{eq: charact vectors 1 not in spec}, we have that
    \begin{align*}
     span\{P_W(A-I)^{-1}b_j:\, 1\leq j\leq L\}&=span\{P_W(A-I)^{-1}\left(-b_j+\sum_{i=1}^{j-1}b_i\right):\, 1\leq j\leq L\}\\
     &=span\{P_W(g_j):\, 1\leq j\leq L\}.
    \end{align*}
    Therefore, Theorems \ref{coro: nec and sufi no eigenval 1} and \ref{th: general test} coincide. Hence,  
     Theorem \ref{th: general test} generalizes Theorem \ref{coro: nec and sufi no eigenval 1}. However,  when $1\notin \sigma(A)$ Condition \eqref{eq: suf cond no eig 1} from Theorem \ref{coro: nec and sufi no eigenval 1} is simpler and easier to use than Condition \eqref {GenCondComp}.   
\end{remark}

Theorem \ref{th: general test} allows us to generalize Corollary \ref{coro: at least K}. 

\begin{corollary}\label{coro: at least K general}
        Let $A:\F^d\to\F^d$ be a linear operator, and $W$ be a $K$-dimensional subspace of $\F^d$. If  $\{b_1,\dots,b_L\}$  form a complete set for Problem \ref{prob: general}, then $L\ge K$.
        %There are needed at least $L=K$ spatial observational vectors $\{b_1,\dots,b_K\}$ to form a complete set for Problem \ref{prob: general}. 
    \end{corollary}

When $\dim (W)=K$, if $1 \notin \sigma(A)$,  one can always obtain a complete set of observational vectors with cardinality $K$ as in Corollary \ref {coro: at least as many as dim W}. However, when $1\in \sigma (A)$, we may  need  $L>K$ observational vectors for completeness. The following example illustrates the situation where $1 \in \sigma(A)$ when $\dim (W)=1$. 

{\begin{example}\label{example: 2 vect}
            Consider the $7\times 7$ matrix in Jordan form
    \begin{equation}\label{example: A}
    A=\begin{bmatrix}
      \begin{matrix}
      1 & 1 & 0 \\
      0 & 1 & 1 \\
      0 & 0 & 1
      \end{matrix}
      &  \bigzero \\
      \bigzero &  
      \begin{matrix}
      1 & 1 & 0 & 0\\
      0 & 1 & 1 & 0\\
      0 & 0 & 1 & 1\\
      0 & 0 & 0 & 1
      \end{matrix}
    \end{bmatrix}
    \end{equation}
   Let $W:=span\{\omega_1\}$, where $\omega_1$  is any standard  vector basis for $\R^7$. For each one of the seven choices for $\omega_1$, we will provide all possible sampling vectors $b_1\in \{e_1,\dots,e_7\}$ that are complete for Problem \ref{prob: general}. Since in this example, $1 \in \sigma(A)$ there are cases where completeness is not achieved using a single observational vector $b_1$. For those situations, we will provide all possible sets of vectors $\{b_1,b_2\}\subset \{e_1,\dots,e_7\}$ that are complete for the source recovery problem \ref{prob: general}. 
    
    We underline that, in this example, we consider only $\omega_1,b_1,b_2\in\{e_1,\dots, e_7\}$. This corresponds to modeling some applications similar to the one depicted in Figure \ref{fig: city}.
    
    In general, by fixing an eigenvalue $\lambda$ of a matrix $A$ and a Jordan block of size $J\times J$ corresponding to $\lambda$, it gives rise to what is called a \textsl{Jordan chain}: A set of linearly independent vectors $v_j$, $j=1, \ldots, J$. On the one hand, the vector $v_J$ satisfying $(A-\lambda I)^J v_J=0$ is the so-called \textsl{lead vector of the chain}. On the other hand, the vector $v_1:=(A-\lambda I)^{J-1} v_J$ is an eigenvector of $A$. For $1\leq j\leq J$, 
    \begin{equation}\label{eq: shift}
     v_{j-1}:=(A-\lambda I)v_{j}.   
    \end{equation}
    In particular, the lead vector $v_J$ generates the chain via multiplication by $A-\lambda I$.
    
    For this example, the only eigenvalue of $A$ is $\lambda=1$ and $A$ has two Jordan blocks. For each block, we will work with the following Jordan chains:
    \begin{itemize}
        \item Top-left $3\times 3$ Jordan block: $v_1=e_1$ (eigenvector), $v_2=e_2$, and  $v_3=e_3$ (lead vector).
        \item Bottom-right  $4\times 4$ Jordan block: $v_1=e_1$ (eigenvector), $v_2=e_2$, $v_3=e_3$, and $v_4=e_4$ (lead vector).
    \end{itemize}
    
    We start by analyzing the cases where $\omega_1$ is an eigenvector of $A$. That is, $\omega_1=e_1$ and $\omega_1=e_4$. In these situations, Condition \eqref{eq: not observable subspace} from Theorem \ref{negative_conclusion} below implies that there is always a complete vector $b_1$ for the source recovery problem since $\omega_1\in \ker(A-I)$.   %condition \eqref{eq: not observable subspace} from Theorem \ref{negative_conclusion}. 
    In fact,  $\omega_1$ is contained in one of the two Jordan chains, and we will show that any vector $b_1$  in that chain is complete. This situation is summarized in the following table: 
    \begin{center}
    \begin{tabular}{||c || c  ||} 
     \hline \hline
    $\omega_1$ & \textbf{complete observational vector $b_1$} \\
    [0.5ex] 
     \hline\hline
    $e_1$      & $e_1$ or $e_2$ or $e_3$              \\ \hline
    $e_4$      & $e_4$ or $e_5$ or $e_6$ or $e_7$ \\ 
    \hline
    \end{tabular}
    \end{center}
    To show this, consider for example the case $\omega_1=e_1$ (the case $\omega_1=e_4$ is analogous). The fact that $b_1=e_1$ is complete is  $\textsl{trivial}$  as it  samples exactly at the source location $\omega_1=e_1$.  For the other two choices $b_1=e_2$, $b_1=e_3$, their minimal $A$-annihilator polynomials are $m_{b_1}(x)=\kappa_1(x)=(x-1)^2$ and $m_{b_1}(x)=\kappa_1(x)=(x-1)^3$, resp. Thus, $\widehat{\kappa_1}$ is given by $(x-1)$ if $b_1=e_2$ and by $(x-1)^2$ if $b_1=e_3$. %
    We now show that Condition \eqref{GenCondComp} from Theorem \ref{th: general test} is satisfied when $b_1=e_2$.  For this case,  the vector $g_1$ in \eqref {eq: characteristic vectors general case} is given by $g_1=(A-I)e_2$, and  we get 
    $$ P_W(g_1)=P_W(A-I)e_2=P_We_1=P_W\omega_1=\omega_1.$$
    Thus, $b_1=e_2$ is complete. Similarly, if $b_1=e_3$, then $g_1=(A-I)^2e_3$ and we have $ P_W(g_1)=\omega_1$, and $b_1=e_3$ is complete.

    Now, if  $\omega_1\in\{e_2,e_3,e_5,e_6,e_7\}$, Theorem \ref{negative_conclusion} below will determine that one spatial observational vector $b_1$ is never complete, that is, it is not enough for solving Problem \ref{prob: general}. 
    Indeed, in these cases $\omega_1$ is orthogonal to $\ker(A-I)=\mathrm{span}\{e_1,e_4\}$, and since $(A^*-I)^7=\textbf{0}$, it trivially holds that $\omega_1\in\ker((A^*-I)^7)$. Thus, $\omega_1$ does not satisfy Condition \eqref{eq: not observable subspace}. Therefore, at least two distinct spatial observational sampling vectors $b_1$ and $b_2$ will be needed  for a complete observational set.
    
    We recall that, for this example, we added the constraint $b_1,b_2\in\{e_1,\dots,e_7\}$. Thus, since $A$ is given in Jordan form, we fall under three possible scenarios:
    \begin{equation}
        Z(A;b_1)\cap Z(A;b_2)=\{\textbf{0}\} \quad \text{or} \quad Z(A;b_2)\subseteq Z(A;b_1) \quad \text{or} \quad Z(A;b_1)\subseteq Z(A;b_2).
    \end{equation}
    Since the last two cases are the same if we switch the roles of $b_1$ and $b_2$, we will just focus on the first two. 
    
  The first case will never give us a complete set for the source recovery problem. To see this,  note that if $Z(A;b_1)\cap Z(A;b_2)=\{\textbf{0}\}$, then $b_1 \in \{e_1,e_2,e_3\}$ and $b_2 \in \{e_4,e_5,e_6, e_7\}$ (or vice versa). Then, the vectors in \eqref {eq: characteristic vectors general case} are  $g_1=\widehat m_{b_1}(A)b_1=e_1$ and $g_2=\widehat m_{b_2}(A)b_2=e_4$ (or vice versa). Hence,  $P_W(g_1)=0$ and $P_W(g_2)=0$, and $\{b_1,b_2\}$ is not complete.  
     
 For the  case $Z(A;b_2)\subseteq Z(A;b_1)$,
 %, let $b_1, b_2$ be such that $b_2\in Z(A;b_1)$. That is, 
  let  $b_2=q_2^1(A)b_1$ for some polynomial $q_2^1\in\F[x]$. Since $b_2\in Z(A;b_1)$, the minimal $A$-conductor  of $b_2$ to $Z(A;b_1)$ is $\kappa_2\equiv 1$. Thus,   $\widehat{\kappa_2}\equiv 0$. Then, the vectors defined in \eqref{eq: characteristic vectors general case} are $g_1=\widehat m_{b_1}(A)b_1$, and $g_2= \widehat{q_2^1}(A)b_1$. %defined \eqref{eq: characteristic vectors general case} in Theorem \ref{coro: nec suff cond} are of the form 
  In particular, $P_W(g_1)=0$. Choosing $b_2=(A-I)\omega_1$, we have that $P_W(g_2)\ne\textbf{ 0}$.  
   Therefore, we obtain the following table summarizing all possible complete spatial observational sets of sampling vectors $\{b_1,b_2\}\subset\{e_1,\dots,e_7\}$ for each choice of the source $\omega_1$: 
    \begin{center}
    \begin{tabular}{||c || c | c ||} 
    \hline\hline
     $\omega_1$ & $b_1$ & $b_2$  \\ [0.5ex] 
     \hline\hline
     $e_2$      &   $e_3$ or $e_2$ & $e_1$                                   \\  \hline
    $e_3$      &        $e_3$ & $e_2$                                      \\  \hline
    $e_5$      &                       $e_7$  or $e_6$ or $e_5$ & $e_4$                       \\  \hline
    $e_6$      &    $e_7$  or $e_6$ & $e_5$                                           \\  \hline
    $e_7$      & $e_7$ & $e_6$             \\           
     \hline
    \end{tabular}
    \end{center}
    \end{example}

 \subsection {Existence of an observational vector for a single source term problem}\label{Single Source Term}

Theorem \ref{th: general test} provides a test for completeness of a given spatial observations set of sampling vectors.  Given $W$ with $\dim W=K$, when $1 \notin \sigma (A)$, a complete set of $K$ observational vectors always exists. However, when $1 \in \sigma (A)$,  Examples \ref {example_no_sol} and  \ref {example: 2 vect}} show   that we may need more than $K$ vectors. 
In order to address the challenges associated with the case where $1\in \sigma(A)$, we make the assumption for the rest of this section that the source term belongs to a one-dimensional subspace $W$ of $\F^d$ spanned by a known unit vector $\omega_1$
and provide the necessary and sufficient conditions under which there exists a complete set for Problem \ref{prob: general} that consists of a single vector.

\begin{theorem}
\label{negative_conclusion}
Let $A$ be a linear operator in $\F^d$, and let  $W$ be the one-dimensional subspace of $\F^d$ spanned by a unit vector $\omega_1\in\F^d\setminus\{{\bf 0}\}$.
Then, there exists a complete observational vector $b$ for Problem \ref{prob: general}  if and only if 
\begin{equation}\label{eq: not observable subspace}
    \omega_1\not\in \ker(A^*-I)^d\cap(\ker (A-I))^\perp= (Range((A-I)^d))^\perp\cap(\ker (A-I))^\perp.
\end{equation}        
\end{theorem}

 \section{poofs}\label{sec: proofs}  
 
 \begin {proof}[Proof of Theorem  \ref{th: prop_characterization_by_eq}]
        Let $L=1$ and write $b_1:=b$. 
       Since  $\Lambda_n$ in \eqref{Lambda_powers of A}  is the sum of consecutive powers of $A$, 
         it is not difficult to show that       \begin{equation}\label{Lemma_frame_frame}
            span\{\Lambda_1 b,\Lambda_2 b,\dots,\Lambda_{N} b\}=span\{b,Ab,\dots,A^{N-1}b \}.
        \end{equation}       
       From its definition  \eqref {eq: span of measurements}, any vector  $[f \; \; g]^T \in \mathcal{O}(b)$ can be written as :
        \begin{equation*}
        \left[ {\begin{array} {r}
        f \\
        g
         \end{array} }\right] =\sum_{n=0}^k a_n \left[ {\begin{array} {r}
        A^nb \\
        \Lambda_nb
         \end{array} }\right],
        \end{equation*}
        for some scalars $a_0,\dots,a_k$. Then, 
        using the simple identity 
        \begin{equation}\label{eq: trick}
          (A-I)\Lambda_k=(A-I)\sum_{j=0}^{k-1}A^j=A^k-I,  
        \end{equation}
        we get that
        \begin{align*}
            (A-I)g=\sum_{n=0}^ka_n(A-I)\Lambda_n b=\sum_{n=0}^ka_n(A^n-I)b
            =f-\sum_{n=0}^ka_n b.
        \end{align*}
        From this last expression, it follows 
        that $$f=(A-I)g+\mu b, \qquad \text{where } \mu=-\sum_{n=0}^ka_n, \text{ and } g=\sum_{n=1}^ka_n\Lambda_nb.$$ 
        From   \eqref{Lemma_frame_frame}, and the expression for $g$, we conclude that $g\in Z(A;b)$.
        Thus, 
        \begin{equation*}
            span \left\{\left[ {\begin{array} {r}
        A^nb \\
        \Lambda_nb
         \end{array} }\right]: \, n\in\N_0\right\} \subseteq \left\{\left[ {\begin{array} {rr}(A-I)g+\mu b\\g\qquad\qquad\end{array} }\right]\in\F^{2d}: \, g\in Z(A;b), \, \mu\in\F \right\}.
        \end{equation*}
        Conversely, consider the vector in $\F^{2d}$ 
        $$\left[ {\begin{array} {rr}(A-I)g+\mu b\\g\qquad \qquad \end{array} }\right]$$
        for an arbitrary vector  $g\in Z(A;b)$ and a scalar $\mu\in\F$. By hypothesis, $g$ is a finite linear combination of vectors of the form $A^nb$. Using \eqref{Lemma_frame_frame} again, it follows that 
        \begin{equation}\label{eq: aux thm 2}
         g=\sum_{n=1}^ka_n\Lambda_n b   
        \end{equation}
        for some scalars $a_1,\dots,a_k$. Now, by utilizing \eqref{eq: trick},
        we have
        \begin{align*}
          (A-I)g&=(A-I)\sum_{n=1}^ka_n\Lambda_n b=\sum_{n=1}^ka_nA^nb-\sum_{n=1}^ka_nb\\
          &=\sum_{n=1}^ka_nA^nb-\sum_{n=1}^ka_nb\pm a_0b =\sum_{n=0}^ka_nA^nb-\sum_{n=0}^ka_nb
        \end{align*}
        where the equality holds for any $a_0\in\F$. In particular, if we choose
        $$a_0:= \mu-\sum_{n=1}^ka_n,$$
        then 
        \begin{equation}\label{eq: aux 2 thm 2}
            (A-I)g+\mu b=\sum_{n=0}^k a_nA^nb-\sum_{n=0}^ka_nb+\sum_{n=0}^ka_nb=\sum_{n=0}^ka_nA^nb.
        \end{equation}
        Since $\Lambda_0={\bf 0}$, the expressions \eqref{eq: aux thm 2} and \eqref{eq: aux 2 thm 2}  prove that
        $$\left[ {\begin{array} {rr}(A-I)g+\mu b\\g\qquad\qquad \end{array} }\right]=
        \sum_{n=0}^ka_n\left[ {\begin{array} {rr}A^nb\\\Lambda_nb\end{array} }\right].
        $$
        Hence,
        \begin{equation*}\left\{\left[ {\begin{array} {rr}(A-I)g+\mu b\\g\qquad\qquad\end{array} }\right]\in\F^{2d}: \, g\in Z(A;b), \, \mu\in\F \right\}\subseteq 
            span \left\{\left[ {\begin{array} {r}
        A^nb \\
        \Lambda_nb
         \end{array} }\right]: \, n\in\N_0\right\}. 
        \end{equation*}    
        
        The case for multiple vectors $b_1,\dots,b_L$ follows by using the same arguments.

       Finally, the dimension of $\mathcal{O}(b_1,\dots,b_L)\subset \F^{2d}$
        is at most  $d+L$ since 
         \begin{equation*}\left[ {\begin{array} {rr}(A-I)g+\sum\limits_{i=1}^L\mu_i b_i \\g\qquad\qquad\end{array} }\right]= \left[ {\begin{array} {cc}(A-I)\\I\end{array} }\right]g+\sum\limits_{i=1}^L \mu_i \left[ {\begin{array} {cc}b_i\\\textbf{0}\end{array} }\right], \quad g \in  Z(A;b_1,\cdots,b_L)\subseteq \F^d.         \end{equation*}     
    \end{proof}

     \begin{proof}[Proof of Corollary \ref{rank_d+1} ]
    Consider the operator $F:\F^d\to\R^{2d}$ given in matrix form by
    \begin{equation}
        F:=\begin{bmatrix}
            A-I\\
            I
        \end{bmatrix}
    \end{equation}
    where $I$ denotes the $d\times d$ identity matrix. Since $\mathrm{dim}(Z(A;b))=r$, where $r$ is the degree of $m_b$ (see Remark \ref {KS}), then $\mathrm{dim}(F(Z(A;b)))\leq r$. On the other hand, using that $\{b,Ab,\dots, A^{r-1}b\}$ is a basis from $Z(A;b)$ (see Remark \ref {KS}), it follows that vectors 
    \begin{equation*}
        F(A^jb)=\begin{bmatrix}
            (A-I)A^jb\\
            A^jb
        \end{bmatrix} \qquad \text{for } \quad 0\leq j\leq r-1
    \end{equation*}
    are linearly independent. So, $\mathrm{dim}(F(Z(A;b)))= r$. Since $b\ne \textbf{0}$, for any non-zero $g\in Z(A;b)$, the vectors
    \begin{equation*}
        \begin{bmatrix}
            (A-I)g\\
            g
        \end{bmatrix} \quad \text{ and } \quad
                \begin{bmatrix}
            b\\
            \textbf{0}
        \end{bmatrix}
    \end{equation*}
    are linearly independent.
%        \begin{equation*}
%          \begin{bmatrix}
%            b\\
%            \textbf{0}
%        \end{bmatrix} \ne F(g) \text { for any } g \in Z(A,b).
%          \end{equation*}  
          Therefore, from the characterization given by Theorem \ref{th: prop_characterization_by_eq}, 
     \begin{align}\label{eq: all samples for 1}
                \mathcal{O}(b):= span \left\{\left[ {\begin{array} {r}
                A^nb \\
                \Lambda_nb
                 \end{array} }\right]: \, n\in\N_0\right\} &= \left\{\left[ {\begin{array} {rr}(A-I)g+\mu b\\ 
                 g\qquad \qquad \end{array} }\right]
                 : \, g\in Z(A;b), \, \mu\in\F \right\}
    \end{align}
    and we have that the dimension of $\mathcal{O}(b)$ is $r+1$.  
    \end{proof}

\subsection {Proofs of Section \ref {subsec: nec suf cond}}
We start with a lemma that relates linear combinations of measurements as in \eqref{data} with vectors in the augmented observational orbits space 
$\mathcal O(b_1,\dots,b_L)$ defined in \eqref {eq: span of measurements}. Specifically, let $\{\alpha_{n,\ell}\}_{n,\ell}$ be an arbitrary finite set of scalars, and consider the linear combination of measurements 
\begin{equation}\label{eq: equations with g a}
            \sum_{n,\ell}\alpha_{n,\ell} \, y_{\ell}(n)= \langle x_0, \sum_{n,\ell}\alpha_{n,\ell}A^n b_{\ell}\rangle +\langle \omega, \sum_{n,\ell}\alpha_{n,\ell}\Lambda_n b_{\ell}\rangle. 
        \end{equation} We have:
 \begin{lemma}\label{lemma aux}
        Let $A$ be a linear operator in $\F^d$, and $b_1,\dots, b_L$ be vectors in $\F^d$.    
        For each $n\in\N_0$ and $\ell\in \{1,\dots, L\}$, let $y_\ell(n):=\langle x_0, A^n b_{\ell}\rangle +\langle \omega, \Lambda_n b_{\ell}\rangle$ as in \eqref{data}. Then 
        \begin{enumerate}
         \item[(a)] For any linear combinations of measurements as in \eqref {eq: equations with g a}, the vectors  
        \begin{equation}\label{eq: vectors f and g}
          f:=\sum\alpha_{n,\ell}A^n b_{\ell} \qquad \text{ and }\qquad  g:=\sum\alpha_{n,\ell}\Lambda_n b_{\ell}.  
        \end{equation}
        are such that, $g\in Z(A;b_1,\dots,b_L)$ and $f = (A-I)g+ \sum_{\ell=1}^L \mu_\ell b_\ell$ for some $\mu_1,...,\mu_L\in \mathbb{F}$ (which depend on the scalars $\alpha_{n,\ell}$). Hence (by Theorem \ref {th: prop_characterization_by_eq}) $(f, g)^T \in \mathcal O(b_1,\dots,b_L)$.    
       \item[(b)] Conversely, for any pair of vectors $f,g$ such that $g\in Z(A;b_1,\dots,b_L)$, $f:= (A-I)g+ \sum_{j=1}^L \mu_\ell b_\ell,$
        where $\mu_1,...,\mu_L\in \mathbb{F}$, 
       there exists a finite set of scalars $\{\alpha_{n,\ell}\}_{n,\ell}\subset \F$ (depending on $\mu_1,\dots,\mu_L$ and $g$) such that $f$ and $g$ can be written as in \eqref{eq: vectors f and g} and Equation \eqref {eq: equations with g a} is satisfied. In particular  
        \begin{equation}\label{eq: equations with g2}
            \sum\alpha_{n,\ell} \, y_{\ell}(n)= \langle x_0, f\rangle +\langle \omega, g\rangle.
        \end{equation} 
        \end{enumerate}
    \end{lemma}
    
   \begin{proof}
        Using linear combinations of the time-space measurements \eqref{data}, we can generate any equation of the form 
        \begin{equation}\label{eq: equations with g}
            \underbrace{\sum\alpha_{n,\ell} \, y_{\ell}(n)}_{\widetilde{y}}= \langle x_0, \underbrace{\sum\alpha_{n,\ell}A^n b_{\ell}}_{f}\rangle +\langle \omega, \underbrace{\sum\alpha_{n,\ell}\Lambda_n b_{\ell}}_{g}\rangle.
        \end{equation}
        From 
        Theorem \ref{th: prop_characterization_by_eq}
        we know that the vectors $f$ and $g$ defined in Equation \eqref{eq: equations with g} are related by $f = (A-I)g+ \sum \mu_\ell b_\ell$ for $g\in Z(A;b_1,\dots,b_\ell)$ and for some scalars $\mu_\ell$,  $\ell=1,\dots,L$. Conversely, since Theorem \ref{th: prop_characterization_by_eq} gives  us the equality
        \begin{align*}
            \mathcal{O}&(b_1,\dots,b_L)\\&=\left\{
            \left[ {\begin{array} {r}
            (A-I)g+\displaystyle\sum_{\ell=1}^L \mu_\ell\, b_\ell \\g\qquad \qquad \qquad \end{array} }\right]\in\F^{2d}: \, g\in Z(A;b_1,\dots,b_L), \, \mu_1,\dots,\mu_L\in\F \right\},
        \end{align*}
        we have that given any vector $g\in Z(A;b_1,\dots,b_\ell)$ and any scalars $\mu_1,\dots,\mu_L$, the vector $f := (A-I)g+ \sum \mu_\ell b_\ell$ belongs to $Z(A;b_1,\dots,b_L)$ and thus there exists scalars $\{\alpha_{n,\ell}\}$ that provides us an equation like \eqref{eq: equations with g}. 
   \end{proof}

    \begin{proof}[Proof of Theorem \ref{coro: nec suff cond}]\, 
        \begin{itemize}
 \item             
        Suppose there exist vectors $g_1,\dots,g_N\in Z(A;b_1,\dots,b_L)$ satisfying Conditions \ref{item: 2nd cond} and \ref{item: 1st cond} of Theorem \ref {coro: nec suff cond}. 
    
       Condition \ref{item: 2nd cond} states that there exists a matrix $M\in\F^{L\times N}$ satisfying equation  \eqref{eq: nec cond 2} $(A-I)G - BM = 0$ (where $B$ and $G$ are the matrices having the vectors $b_1,\dots,b_L$ and the vectors $g_1,\dots, g_N$, resp.,  as columns). Denoting $\mu_\ell^j:=-M(\ell,j)$, we can rewrite \eqref{eq: nec cond 2} as $N$ equations of the form $(A-I)g_j+ \sum_{\ell=1}^L \mu_\ell^j b_\ell={\bf 0}$, for $1\leq j\leq N$. Defining $f_j := (A-I)g_j+ \sum \mu_\ell^j b_\ell$, we have that $f_j=\textbf{0}$ for each $1\leq j\leq N$. Then,  Part (b) of Lemma \ref{lemma aux}  implies that, for each $1\leq j\leq N$ 
        there exists a finite set of scalars $\{\alpha_{n,\ell}^j\}_{n,\ell}\subset \F$ (that depend on $\mu_1,\dots,\mu_L$ and $g_j$) such that         
        we get $N$ equations of the form
        \begin{equation}\label{eq: the equations that we need}
            \sum_{n,\ell} \alpha_{n,\ell}^j \, y_\ell(n) = \langle \omega,f_j\rangle+ \langle \omega,g_j\rangle, \qquad 1\leq j\leq N.
        \end{equation}
        Using the fact that $f_j=\textbf{0}$, and  that $\omega\in W$, we obtain
        \begin{equation}\label{eq: the equations that we need2zz}
            \sum_{n,\ell} \alpha_{n,\ell}^j \, y_\ell(n) = \langle \omega,g_j\rangle=\langle P_W \omega, g_j\rangle= \langle  \omega, P_W g_j\rangle, \qquad 1\leq j\leq N.
        \end{equation}

        Since the set of vectors $\{g_1,\dots,g_N\}$ also satisfies Condition \ref{item: 1st cond}, $\{P_Wg_1,\dots,P_Wg_N\}$ is a spanning set for $W$. Thus, $\{P_Wg_1,\dots,P_Wg_N\}$ form a frame for $W$ and $\omega$ can be recovered from the system of equations \eqref {eq: the equations that we need}.

        \item For the converse, assume that there exists a set of spatial observational vectors $\{b_1,\dots,b_L\}$ that are complete for Problem \ref{prob: general}. That is, there exists a finite number $N\geq K$ of equations of the form \eqref{eq: equations with g a} that allows us to solve for $\omega\in W$ using linear operations on the measurements $y_{\ell}(n)$. Due to Lemma \ref{lemma aux}, such a system of $N$ equations of the form \eqref{eq: equations with g a} can be written as a system of $N$ equations of the form \eqref{eq: equations with g2} and, in matrix notation, we can rewrite the hypothesis that $\{b_1,\dots,b_L\}$ are complete for Problem \ref{prob: general} as follows:      
        There exist $N$ vectors $(f_1,g_1)^T,\dots, (f_N,g_N)^T\in \mathcal{O}(b_1,\dots,b_L)$  and a vector $\widetilde{y}:=(\widetilde{y}_1,\dots,\widetilde{y}_N)^T\in \F^N$  where $\widetilde{y}_j:=\sum\limits_{\ell=1}^L\sum_{n}\alpha_{n,\ell}^j \, y_{\ell}(n)$ for a finite set of scalars $\{\alpha_{n,\ell}^j\}$, such that any $\omega\in W\subseteq\F^d$ can be recovered from the  matrix system 
        \begin{equation*}
            \begin{bmatrix}
                F^* & G^*
            \end{bmatrix}\begin{bmatrix}
                x_0\\
                \omega
            \end{bmatrix}=\widetilde{y},
            %\begin{bmatrix}
            %     \widetilde{y}_1\\
            %     \widetilde{y}_2\\
            %     \vdots\\
            %     \widetilde{y}_T
            % \end{bmatrix},
        \end{equation*}
        where $F^*$ and $G^*$ are the adjoints of the $d\times N$ matrices $F,G$ having  the vectors $f_1,\dots,f_N$ and $g_1,\dots,g_N$ as columns, respectively.
        This means that there exists an $N\times d$ matrix $U$ such that
        \begin{equation}\label{eq: matrix 33}
          \omega= U^* \begin{bmatrix}
            F^* & G^*
        \end{bmatrix}\begin{bmatrix}
            x_0 \\ \omega
        \end{bmatrix}=U^*
                \widetilde{y} . 
        \end{equation}
           %Since the map $y\mapsto U^* y$ is linear, the right equality gives a linear system that is made of equations of the form
           Since the map $y\mapsto U^* y$ is linear, the right-hand side $U^*\widetilde{y}$
        is another linear combination of the samples $\{y_\ell(n)\}$ of the dynamical system \eqref{eq: dyn system subject to}. Hence, \eqref{eq: matrix 33} is a system of equations of the form \eqref{eq: equations with g a}. 
        Rewriting the left inequality in \eqref{eq: matrix 33}, we obtain that
        \begin{equation}\label{eq: FU GU}
            \omega=(FU)^*x_0+(GU)^*\omega \qquad \text{ for any } x_0\in \F^d, \text{ and any } \omega\in W.
        \end{equation}
        Let us denote by $g_1',\dots, g_d'$ and $f_1',\dots, f_d$ the column vectors of the $d\times d$ matrices $G^\prime=GU$ and $F^\prime=FU$, respectively. Then, by Lemma \ref{lemma aux},  for each $1\leq j\leq d$, there exist scalars $\mu_\ell^j$, ($1\leq\ell\leq L$) such that 
            $$f_j' = (A-I)g_j'+ \sum_{\ell=1}^L \mu_\ell^j b_\ell^j.$$

        We now show that $\{P_W(g_1'),\dots,P_W(g_d')\}$ spans $W$ and $f_j' \equiv 0$. Taking $x_0=\textbf{0}$ in \eqref{eq: FU GU}, we get that 
        $\omega=(GU)^*\omega$ for any $\omega \in W$ or, equivalently,
        $$P_W(v) = (GU)^*P_W(v) \qquad \forall v\in\F^d.$$
        This means that $P_WGU=PG^\prime=P_W$. Therefore,  $\{P_W(g_1'),\dots,P_W(g_d')\}$ spans $W$. That is, the vectors $g_1',\dots, g_d'$ satisfy Condition \ref{item: 1st cond} for $N=d$.
        
        Finally, to show that Condition \ref{item: 2nd cond} is satisfied, note that since for any $x_0\in \F^d$,  and any  $\omega\in W$
        $$\omega=(FU)^*x_0+(GU)^*\omega \quad \text{ and } \quad \omega=(GU)^*\omega ,$$
        we have $(FU)^*=(F^\prime)^*\equiv\textbf{0}$.  Hence, $f_j'=(A-I)g_j'+ \sum_{\ell=1}^L \mu_\ell^j b_\ell^j={\bf 0}$. Thus, letting $M$ be the $L\times d$ matrix given  by $M(\ell,j)=-\mu_\ell^j$, we  and obtain that the vectors $g_1',\dots, g_d'$ also satisfy Condition \ref{item: 2nd cond}.
        
\end{itemize}

    \end{proof}

\subsection {Proofs of Section \ref {subsubsec: not 1}} To prove Theorem \ref{coro: nec and sufi no eigenval 1}, we need the preliminary results given by Lemma    \ref {lem: inv I-A in orbit} and  Proposition \ref {remark: g when A not eigenval 1} below.
  
    \begin{lemma}\label{lem: inv I-A in orbit}
          Let $A:\F^{d}\to \F^d$ be a linear operator, let $b\in \F^d$, and consider the function $f(x)=1/(x-1)$. If the minimal $A$-annihilating polynomial $m_b(x)$ of $b$ is such that $m_b(1)\not=0$, then 
          $f(A)b$ is well-defined and   $f(A)b\in Z(A;b)$. %%, and we will write  $$(I-A)^{-1}b\in Z(A;b).$$ 
          In particular, if $A$ does not have eigenvalue $1$, then, for any $b\in\F^d$, $(A-I)^{-1}b\in Z(A;b).$ 
    \end{lemma}
    
    \begin{proof}
        This result follows by applying Theorem \ref{th: interpol th} in the Appendix to the function $f(x)=1/(1-x)$ and its Hermite interpolating polynomial $q$ on the roots of $m_b(x)$  given by \eqref{eq: interpolating pol}. Recall that $p(A)b\in Z(A;b)$ for every polynomial $p\in\F[x]$.
    \end{proof}
For the next proposition, we use the same notation as in Theorem \ref {coro: nec suff cond}. 
      \begin{proposition}\label{remark: g when A not eigenval 1}
         Let $A:\F^{d}\to \F^d$ be a linear operator such that $1\notin \sigma (A)$, and  let $\{b_1,\dots,b_L\}$ be a set of vectors in $\F^d$. If vectors $g_1,\dots,g_T\in \F^d$ satisfy Condition \ref{item: 2nd cond} of Theorem \ref{coro: nec suff cond}, then they can be written as the columns of $G:=(A-I)^{-1}BM$. Moreover,  $g_1,\dots,g_T\in Z(A;b_1,\dots, b_L)$. 
    \end{proposition}
\begin{proof}
   Since $A$ does not have eigenvalue $1$, $G=(A-I)^{-1}BM$ is well-defined and from Lemma \ref{lem: inv I-A in orbit} its columns belong to $Z(A;b_1,\dots, b_L)$.
\end{proof}

We now prove Theorem  \ref {coro: nec and sufi no eigenval 1} using the two previous results.
    \begin{proof}[Proof of Theorem \ref{coro: nec and sufi no eigenval 1}]
 The proof of the necessity of the equality \eqref{eq: suf cond no eig 1} is similar to the proof of Corollary \ref{coro: at least K} below. To prove sufficiency, assume that \eqref{eq: suf cond no eig 1} is satisfied and consider the matrix $G := (A-I)^{-1}B$. Then, clearly, \eqref{eq: nec cond 2} is satisfied, with $M$ being the $L\times L$ identity matrix. The columns of $G$ are the vectors $g_i := (A-I)^{-1}b_i$ for $i = 1, \dots, K$. By Lemma \ref{lem: inv I-A in orbit}, for each $i = 1, \dots, K$, we have that $g_i \in Z(A; b_i)$. By equality \eqref{eq: suf cond no eig 1}, $P_WG$ has rank $K$, so its columns, which are the vectors $P_Wg_i$, span $W$. Thus, from Theorem \ref{coro: nec suff cond}, the observational vectors $\{b_1, \dots, b_L\} \subset \F^d\setminus{\{\bf 0\}}$ are complete for Problem \ref{prob: general}.
        
    \end{proof}

    \begin{proof}[Proof of Corollary \ref{coro: at least K}]
        Let $G$ and $B$ be as in the statement of Theorem \ref{coro: nec suff cond}. Condition \ref{item: 1st cond} states that $rank(P_WG) = K$, which implies that $\rank(G) \geq K$. Thus, using Equation \eqref{eq: nec cond 2}, we have that $\rank(G) = \rank\big((A-I)^{-1}BM\big) \geq K$. Therefore, $\rank(B) \geq K$. Consequently, we require at least $K$ linearly independent spatial observational vectors among ${b_1, \dots, b_L}$, implying $L \geq K$.  
    \end{proof}

\begin{proof}[Proof of Corollary \ref{coro: at least as many as dim W}]
        Assume that $A-I$ is a non-singular operator in $\F^d$, and that the source term in \eqref{eq: dyn system subject to} is such that $\omega\in W$ where $\mathrm{dim}(W)=K$. One strategy for positioning  the sensors $b_i$ is the following: 
        Consider  $\{\omega_\ell\}_{\ell=1}^K$ an orthonormal basis for $W$, and then define the vectors 
        \begin{equation}\label{eq: particular sampling vect}
           b_\ell:=(I-A)\omega_\ell, \qquad \ell\in\{1,\dots, K\}.  
        \end{equation}
        For each $\ell\in\{1,\dots, K\}$, consider the scalars $\{\alpha_{n,\ell}\}_{n=0}^{r_\ell}$ that are the coefficients of  the {minimal} $A$-annihilating polynomial $m_{b_\ell}(x)=\sum_{n=0}^{r_\ell}\alpha_{n,\ell}x^n$ of $b_\ell$. Then, given the measurements 
        $$y_\ell(n)=\langle x_0, A^n b_\ell\rangle +\langle \omega, \Lambda_nb_\ell\rangle$$
        as in \eqref{data}, for each $\ell=1,\dots, K$ we  use Identity \eqref {eq: trick} and the fact that $m_{b_\ell}(A)b_{\ell}= \textbf {0}$ to get
        \begin{align*}
             \sum_{n=0}^{r_\ell}\alpha_{n,\ell} y_\ell(n)&=\langle \omega, \sum_{n=0}^{r_\ell}\alpha_{n,\ell}\Lambda_n b_\ell\rangle= m_{b_\ell}(1)\langle \omega,(I-A)^{-1}b_\ell\rangle= m_{b_\ell}(1)\langle \omega,\omega_\ell\rangle.   
        \end{align*}
        Hence, since $1 \notin\sigma(A)$, $m_{b_\ell}(1)\ne 0$, we can solve for $\langle \omega,\omega_\ell\rangle$ above and obtain 
        $$\omega=\left(\frac{1}{m_{b_1}(1)}\sum_{n=0}^{r_1}\alpha_{n,1} y_1(n)\right)\omega_1+\dots +\left(\frac{1}{m_{b_K}(1)}\sum_{n=0}^{r_K}\alpha_{n,K} y_K(n) \right)\omega_K.$$
        Moreover, if $r_j$ is the degree of the minimal $A$-annihilating polynomial of $b_j$, then the  number of time samples that are needed to find $\omega$ is $T=\max\limits_{\ell\in\{1,\dots, K\}}\{r_\ell+1\}$. %During are enough to recover $\omega\in W$, and 
\end{proof}

\subsection {Proofs for Section \ref {subsub: general}}
Now we will prove Theorem \ref{th: general test}. To do so we need several preliminary results and their proofs stated in Lemma \ref {lemma: lemma very auxiliar}, Theorem \ref {th: characterization of all g}, and Corollary \ref {coro: dim L} below.
\begin{lemma}\label{lemma: lemma very auxiliar}
  Using Notation  \ref {HatNotation}, we have the following.
  \begin{enumerate} 
      \item For $u(x),v(x)\in \F[x]$, we have \begin{equation}\label{eq: eq like a lemma}
            \widehat{uv}(x)=\widehat{u}(x)v(x)+u(1)\widehat{v}(x).
        \end{equation}
    \item Let $A:\F^{d}\to\F^d$ be a linear operator,   $b\in \F^d$, and $m_b$ be its minimal $A$-annihilating polynomial. If $p\in\F[x]$ is such that $p(A)b=\textbf{0}$, then there exists $a \in \F$  $$\widehat{p}(A)b=a \, \widehat{m_b}(A)b. $$
  \end{enumerate}
\end{lemma}    
\begin{proof}
    Identity \eqref {eq: eq like a lemma} can be established by a straightforward computation.  
    
    To prove Part (2) of the Lemma, we note that, since by hypothesis $p$ is an $A$-annihilating polynomial of $b$, $m_b$ must divide $p$.  Thus, there exists $q\in\F[x]$ such that $p(x)=q(x)m_b(x)$. Then, using Part (1) of the lemma with $q=u$,  $m_b=v$,  using the fact that $m_b(A)b=\textbf{0}$, and setting $a=q(1)$ we get
  %       \widehat{p}(x)&=\frac{q(x)m_b(x)-q(1)m_b(1)}{x-1}\\
  %      &=\frac{q(x)-q(1)}{x-1}m_b(x)+q(1)\frac{m_b(x)-m_b(1)}{x-1}.
 %   \end{align*}
%    Hence, as $m_b(A)b=\textbf{0}$, we obtain 
    \begin{equation*}
        \widehat{p}(A)b=\widehat{q}(A)m_b(A)b+q(1)\widehat{m_b}(A)b=q(1)\widehat{m_b}(A)b.
    \end{equation*}
\end{proof}

\begin{definition}\label{def: general admissible vectors}
    Let $A:\F^{d}\to\F^d$ be a linear operator and let $b_1,\dots,b_L\in \F^d$.  We say that  $g\in \F^d$ is an \textit{  $A$-$\{b_1,\dots,b_L\}$ characteristic vector}  if  $g\in Z(A;b_1,\dots, b_L)$ and satisfies
    \begin{equation}\label{eq: condition for g}
     (A-I)g=\sum_{\ell=1}^L\mu_\ell b_\ell   
    \end{equation}
    for some scalars $\mu_1,\dots,\mu_L$. In other words, $g$ is $A$-$\{b_1,\dots,b_L\}$ characteristic  if it belongs to the pre-image 
    $$\mathcal{G}:=\{g\in Z(A;b_1,\dots, b_L): (A-I)g \in span \{b_1,\dots, b_l\}\}.$$
    %$$\mathcal{G}:=span\{g\in \F^d: \, g \text{ is } (b_1,\dots,b_L)\text{-admissible}\}$$
The space $\mathcal G$ will be called    the \textit{$A$-$\{b_1,\dots, b_l\}$ characteristic space of the observational space $span \{b_1,\dots, b_l\}$.} \end{definition}

Inspired by the \textit{Cyclic Decomposition Theorem}, and using Definition \ref {AC} of $A$-conductors, we state and prove the following result.

\begin{theorem}\label{th: characterization of all g}
    Let $A:\F^{d}\to\F^d$ be a linear operator and let $b_1,\dots,b_L\in \F^d$ be non-trivial vectors.
    For each $1\leq j\leq L$ consider $\kappa_j\in \F[x]$ the minimal $A$-conductor polynomial of $b_j$ into the observational orbits space  $V_j:=\sum_{i=1}^{j-1}Z(A;b_i)$ (where $V_1:=\{\textbf{0}\}$, and $\kappa_1:=m_{b_1}$).
    Also, for each $2\leq j\leq L$, let $q_j^i\in\F[x]$, for $1\leq i\leq j-1$, be such that 
    \begin{equation*}%\label{eq: q polynomials}
     \kappa_j(A)b_j=q_j^1(A)b_1+\dots+q_j^{j-1}(A)b_{j-1}.   
    \end{equation*}
    Then, $g$ is an $A$-$\{b_1,\dots,b_L\}$ characteristic vector, if and only if it is of the form
    \begin{equation}\label{eq: g in general}
      g=\sum_{j=1}^L a_j\left(\widehat{\kappa_j}(A)b_j-\sum_{i=1}^{j-1}\widehat{q_j^i}(A)b_i\right)  
    \end{equation}
    for some scalars $a_1,\dots ,a_L\in\F$.
\end{theorem}
   
\begin{proof}
 First, let us prove that the scalars $a_1,\dots,a_L$ in \eqref{eq: g in general} can be chosen arbitrarily, i.e.,  that any vector $g$ of the form \eqref{eq: g in general} is an $A$-$\{b_1,\dots,b_L\}$ characteristic vector. Indeed, the expression \eqref{eq: g in general} defines a vector in $Z(A;b_1,\dots,b_L)$ as it is a linear combination of polynomials on $A$ applied to the vectors $b_1,\dots,b_L$. Also, by using that $\kappa_j(A)b_j=\sum_{i=1}^{j-1}q_j^i(A)b_i$, for $2\leq j\leq L$,
        \begin{align*}
            (A-I)g&=(A-I)\left(a_1\widehat{m_{b_1}}(A)b_1+\sum_{j=2}^L a_j\left(\widehat{\kappa_j}(A)b_j-\sum_{i=1}^{j-1}\widehat{q_j^i}(A)b_i\right)\right)\\
            &=a_1(m_{b_1}(A)-m_{b_1}(1))b_1+\sum_{j=2}^L a_j\left((\kappa_j(A)-\kappa_j(1))b_j-\sum_{i=1}^{j-1}(q_j^i(A)-q_j^i(1))b_i\right)\\
            &=a_1 m_{b_1}(1)b_1+\sum_{j=2}^L a_j\left(\sum_{i=1}^{j-1}q_j^i(1)-\kappa_j(1)\right)b_j,
        \end{align*}
        That is, if $g$ is given by \eqref{eq: g in general} for some scalars $a_1,\dots,a_L$, then $g$ satisfies \eqref{eq: nec cond 2} with $\mu_1:=a_1 m_{b_1}(1)$, and $\mu_j=a_j\left(\sum_{i=1}^{j-1}q_j^i(1)-\kappa_j(1)\right)$ for $2\leq j\leq L$.

Now, given  an $A$-$\{b_1,\dots,b_L\}$ characteristic vector $g$, we will prove that $g$ can be written in the form \eqref{eq: g in general}. As $g\in Z(A;b_1,\dots,b_L)$, it can be written as
\begin{equation}\label{eq: g eq 1}
  g=p_1(A)b_1+p_2(A)b_2+\dots+p_L(A)b_L  
\end{equation}
for some $p_1,p_2,\dots,p_L\in\F[x]$. Also, there exist $\mu_1,\dots,\mu_L\in\F$ such that
$(A-I)g=\sum_{\ell=1}^L\mu_\ell b_\ell$ or, equivalently, 
\begin{equation}\label{eq: g eq 2}
    \sum_{\ell=1}^L \left((A-I)p_\ell(A)-\mu_\ell\right)b_\ell=\textbf{0}.
\end{equation}
To illustrate how to characterize the polynomials $p_1,\dots,p_L$ satisfying \eqref{eq: g eq 2} for some scalars $\mu_1,\dots,\mu_L$,  we will proceed by cases. 
    \begin{enumerate}
        \item($L=1$): In this case the expression \eqref{eq: g eq 1} reads as $g=p_1(A)b_1$ and the identity \eqref{eq: g eq 2} reads as $((A-I)p_1(A)-\mu_1)b_1=\textbf{0}$. Thus, $m_{b_1}$ divides the polynomial $(x-1)p_1(x)-\mu_1$. That is, there exists a polynomial $\alpha_1(x)\in\F[x]$ such that 
        $$\alpha_1(x)m_{b_1}(x)=(x-1)p_1(x)-\mu_1.$$
        Evaluating that expression at $x=1$, we obtain $\mu_1=\alpha_1(1)m_{b_1}(1)$, and so
        \begin{equation*}
            p_1(x)=\frac{\alpha_1(x)m_{b_1}(x)-\alpha_1(1)m_{b_1}(1))}{x-1}.
        \end{equation*}
        Finally, by using Lemma \ref{lemma: lemma very auxiliar},  the fact that $m_{b_1}(A)b_1=0$, and denoting $a_1=\alpha_1(1)$, we have that
        \begin{equation*}
            g=a_1 \, \widehat{m_{b_1}}(A)b_1.
        \end{equation*}
        This concludes the proof for this case since in the statement of the theorem $\kappa_1:=m_{b_1}$.
        
        \item ($L=2$): In this case the expression \eqref{eq: g eq 1} reads as 
        \begin{equation}\label{eq: g eq 1 case 2 vect}
          g=p_1(A)b_1+p_2(A)b_2  
        \end{equation}
        and the identity \eqref{eq: g eq 2} reads as
        \begin{equation}\label{eq: aux case 2 vect in general}
            ((A-I)p_1(A)-\mu_1)b_1+((A-I)p_2(A)-\mu_2)b_2=\textbf{0}.
        \end{equation} 
        From the last expression we can deduce that $((A-I)p_2(A)-\mu_2)b_2\in Z(A;b_1)\cap Z(A;b_2)$.
        Let $\kappa_2\in\F[x]$ be the generator of polynomial ideal  $\mathbb{I}_2:=\{k\in\F[x]: \, k(A)b_2\in  Z(A;b_1)\}$. Then,
        $\kappa_2$ divides the polynomial $(x-1)p_2(x)-\mu_2$. That is, there exists a polynomial $\alpha_2(x)\in \F[x]$ such that
        \begin{equation}\label{eq: p2 to replace}
            \alpha_2(x)\kappa_2(x)=(x-1)p_2(x)-\mu_2
        \end{equation}        
        Evaluating  the expression above at $x=1$  and solving for $\mu_2$, we obtain
        \begin{equation}\label{eq: p2}
            p_2(x)=\frac{\alpha_2(x)\kappa_2(x)-\alpha_2(1)\kappa_2(1)}{x-1}=:\widehat{\alpha_2\kappa_2}(x).
        \end{equation}
        Replacing \eqref{eq: p2 to replace} into \eqref{eq: aux case 2 vect in general} and taking into account that $\kappa_2(A)b_2=q_2^1(A)b_1$,  we get
        \begin{align*}
            \textbf{0}&=
            ((A-I)p_1(A)-\mu_1)b_1+\alpha_2(A)\kappa_2(A)b_2 \\
            &=\left((A-I)p_1(A)-\mu_1+\alpha_2(A)q_2^1(A)\right)b_1.
        \end{align*}
        Hence, $m_{b_1}$ divides the polynomial $(x-1)p_1(x)-\mu_1+\alpha_2(x)q_2^1(x)$, that is, there exists $\alpha_1(x)\in\F[x]$ such that
        $$\alpha_1(x)m_{b_1}(x)=(x-1)p_1(x)-\mu_1+\alpha_2(x)q_2^1(x).$$
        Evaluating at $x=1$,  we can solve for the value of $\mu_1$ and  obtain
        \begin{equation}\label{eq: p1}
            p_1(x)=\frac{\alpha_1(x)m_{b_1}(x)-\alpha_1(1)m_{b_1}(1)}{x-1}-\frac{\alpha_2(x)q_2^1(x)-\alpha_2(1)q_2^1(1)}{x-1}=\widehat{\alpha_1 m_{b_1}}(x)-\widehat{\alpha_2 q_2^1}(x).
        \end{equation}
        Now, we replace \eqref{eq: p1} and \eqref{eq: p2} into \eqref{eq: g eq 1 case 2 vect} and get
        \begin{equation}\label{eq: g intermediate step}
            g=\widehat{\alpha_1 m_{b_1}}(A)b_1-\widehat{\alpha_2 q_2^1}(A)b_1+\widehat{\alpha_2 \kappa_2}(A)b_2.
        \end{equation}
        Using Lemma  \ref {lemma: lemma very auxiliar},  the fact that $m_{b_1}(A)b_1=0$, we have the identity
        \begin{equation}\label{eq: aux to replace 1}
            \widehat{\alpha_1 m_{b_1}}(A)b_1=\alpha_1(1)\widehat{ m_{b_1}}(A)b_1.
        \end{equation}
        Using Lemma \ref {lemma: lemma very auxiliar} again, we rewrite the polynomials $\widehat{\alpha_2 q_2^1}$ and $\widehat{\alpha_2 \kappa_2}$ as
        \begin{align*}
            \widehat{\alpha_2 q_2^1}(x)
            &=\widehat{\alpha_2}(x)q_2^1(x)+\alpha_2(1)\widehat{q_2^1}(x) \\ 
             \widehat{\alpha_2 \kappa_2}(x)&=\widehat{\alpha_2}(x)\kappa_2(x)+\alpha_2(1)\widehat{\kappa_2}(x),
        \end{align*}
        % and, analogously,
        % \begin{align*}
        %     \widehat{\alpha_2 \kappa_2}(x)&=\widehat{\alpha_2}(x)\kappa_2(x)+\alpha_2(1)\widehat{\kappa_2}(x),
        % \end{align*}
        and use the fact that $\kappa_2(A)b_2=q_2^1(A)b_1$ to obtain 
        \begin{align}\label{eq: aux to replace}
            -\widehat{\alpha_2 q_2^1}(A)b_1+\widehat{\alpha_2 \kappa_2}(A)b_2=-\alpha_2(1)\widehat{q_2^1}(A)b_1+\alpha_2(1)\widehat{\kappa_2}(A)b_2.
        \end{align}
         Finally, by denoting $\kappa_1:=m_{b_1}$, $a_1:=\alpha_1(1)$ and $a_2:=\alpha_2(1)$, we use \eqref{eq: aux to replace 1} and \eqref{eq: aux to replace} into \eqref{eq: g intermediate step} to  obtain
        \begin{align*}
            g=a_1 \, \widehat{\kappa_1}(A)b_1+a_2\left(\widehat{\kappa_2}(A)b_2-\widehat{q_2^1}(A)b_1\right).
        \end{align*}
        \item ($L=3$): In this case the expression \eqref{eq: g eq 1} reads as 
        \begin{equation}\label{eq: g eq 1 case 3 vect}
          g=p_1(A)b_1+p_2(A)b_2  +p_3(A)b_3
        \end{equation}
        and the identity \eqref{eq: g eq 2} reads as
        \begin{equation}\label{eq: aux case 3 vect in general}
            ((A-I)p_1(A)-\mu_1)b_1+((A-I)p_2(A)-\mu_2)b_2+((A-I)p_3(A)-\mu_3)b_3=\textbf{0}.
        \end{equation} 
        From the last expression we can deduce that $((A-I)p_3(A)-\mu_3)b_3\in Z(A;b_1)+ Z(A;b_2)$.
        Let $\kappa_3\in\F[x]$ be the generator of polynomial ideal  $\mathbb{I}_3:=\{k\in\F[x]: \, k(A)b_3\in  Z(A;b_1)+Z(A;b_2)\}$. Then,
        $\kappa_3$ divides the polynomial $(x-1)p_3(x)-\mu_3$. That is, there exists a polynomial $\alpha_3(x)\in \F[x]$ such that
        \begin{equation}\label{eq: p3 to replace}
            \alpha_3(x)\kappa_3(x)=(x-1)p_3(x)-\mu_3
        \end{equation}        
        Similarly to the previous cases, we can obtain
        \begin{equation}\label{eq: p3}
            p_3(x)=\frac{\alpha_3(x)\kappa_3(x)-\alpha_3(1)\kappa_3(1)}{x-1}=:\widehat{\alpha_3\kappa_3}(x).
        \end{equation}
        Replacing \eqref{eq: p3 to replace} into \eqref{eq: aux case 3 vect in general}, and using that $\kappa_3(A)b_3=q_3^1(A)b_1+q_3^2(A)b_2$, we get
        \begin{align}\label{eq: auxx}
            ((A-I)p_1(A)-\mu_1+\alpha_3(A)q_3^1(A))b_1+            ((A-I)p_1(A)-\mu_1+\alpha_3(A)q_3^2(A))b_2=\textbf{0}.
        \end{align}
        From the last expression we can deduce that $((A-I)p_2(A)-\mu_2+\alpha_3(A)q_3^2(A))b_2\in Z(A;b_1)$.
        Let $\kappa_2\in\F[x]$ be the generator of polynomial ideal  $\mathbb{I}_2:=\{k\in\F[x]: \, k(A)b_2\in  Z(A;b_1)\}$. Then,
        $\kappa_2$ divides the polynomial $(x-1)p_2(x)-\mu_2+\alpha_3(x)q_3^2(x)$. That is, there exists a polynomial $\alpha_2(x)\in \F[x]$ such that
        \begin{equation}\label{eq: p2 to replace case 3}
            \alpha_2(x)\kappa_2(x)=(x-1)p_2(x)-\mu_2+\alpha_3(x)q_3^2(x).
        \end{equation}        
        As before, we can deduce
        \begin{equation}\label{eq: p2 case 3}
            p_2(x)=\frac{\alpha_2(x)\kappa_2(x)-\alpha_2(1)\kappa_2(1)}{x-1}-\frac{\alpha_3(x)q_3^2(x)-\alpha_3(1)q_3^2(1)}{x-1}=\widehat{\alpha_2\kappa_2}(x)-\widehat{\alpha_3 q_3^2}(x).
        \end{equation}
        Replacing \eqref{eq: p3 to replace} and \eqref{eq: p2 to replace case 3} into \eqref{eq: auxx}, and using that $\kappa_2(A)b_2=q_2^1(A)b_1$, we get
        \begin{align*}
            ((A-I)p_1(A)-\mu_1+\alpha_2(A)q_2^1(A)+\alpha_3(A)q_3^1(A))b_1=\textbf{0}.
        \end{align*}     
        Hence, $m_{b_1}$ divides the polynomial $(x-1)p_1(x)-\mu_1+\alpha_2(x)q_2^1(x)+\alpha_3(x)q_3^1(x)$, that is, there exists $\alpha_1(x)\in\F[x]$ such that
        \begin{equation}\label{eq: }
          \alpha_1(x)m_{b_1}(x)=(x-1)p_1(x)-\mu_1+\alpha_2(x)q_2^1(x)+\alpha_3(x)q_3^1(x).  
        \end{equation}
        Once again, evaluating at $x=1$, we can solve for the value of $\mu_1$ and then obtain
        \begin{align}\label{eq: p1 case 3}
            p_1(x)&=\frac{\alpha_1(x)m_{b_1}(x)-\alpha_1(1)m_{b_1}(1)}{x-1}-\frac{\alpha_2(x)q_2^1(x)-\alpha_2(1)q_2^1(1)}{x-1}\notag\\
            &\quad -\frac{\alpha_3(x)q_3^1(x)-\alpha_3(1)q_3^1(1)}{x-1}\notag\\
            &=\widehat{\alpha_1 m_{b_1}}(x)-\widehat{\alpha_2 q_2^1}(x)-\widehat{\alpha_3 q_3^1}(x).
        \end{align}
        Now, we replace \eqref{eq: p1 case 3} and \eqref{eq: p2 case 3} into \eqref{eq: g eq 1 case 3 vect} and get
        \begin{align}\label{eq: g intermediate step case 3}
            g&=\left(\widehat{\alpha_1 m_{b_1}}(A)-\widehat{\alpha_2 q_2^1}(A)-\widehat{\alpha_3 q_3^1}(A)\right)b_1+\left(\widehat{\alpha_2 \kappa_2}(A)-\widehat{\alpha_3 q_3^2}(A)\right)b_2+\widehat{\alpha_3 \kappa_3}(A)b_3\notag\\
            &=\underbrace{\widehat{\alpha_1 m_{b_1}}(A)b_1}_{(I)}+\underbrace{\left(\widehat{\alpha_2 \kappa_2}(A)b_2-\widehat{\alpha_2 q_2^1}(A)b_1\right)}_{(II)}+\underbrace{\left(\widehat{\alpha_3 \kappa_3}(A)b_3- \widehat{\alpha_3 q_3^1}(A)b_1-\widehat{\alpha_3 q_3^2}(A)b_2\right)}_{(III)}.
        \end{align}
         By applying  Lemma \ref {lemma: lemma very auxiliar} on the polynomials that appear in the terms $(I)$, $(II)$ and $(III)$ of \eqref{eq: g intermediate step case 3}, together with the identities $\kappa_2(A)b_2=q_2^1(A)b_1$ and $\kappa_3(A)b_3=q_3^1(A)b_1+q_3^2(A)b_2$, we can reach,
        \begin{align*}
            &(I)=\alpha_1(1)\widehat{m_{b_1}}(A)b_1;\\
            &(II)=\cancel{\widehat{\alpha_2}(A) \kappa_2(A)b_2}+\alpha_2(1) \widehat{\kappa_2(A)}b_2-\cancel{\widehat{\alpha_2}(A) q_2^1(A)b_1}-\alpha_2(1) \widehat{q_2^1}(A)b_1;\\
            &(III)=\cancel{\widehat{\alpha_3}(A) \kappa_3(A)b_3}+\alpha_3(1) \widehat{\kappa_3}(A)b_3\\
            &\qquad -\left( \cancel{\widehat{\alpha_3}(A) q_3^1(A)b_1}+\alpha_3(1) \widehat{q_3^1}(A)b_1+\cancel{\widehat{\alpha_3}(A) q_3^2(A)b_2}+\alpha_3(1)\widehat{ q_3^2}(A)b_2\right).
        \end{align*}
        
        Finally, using the notation $\kappa_1:=m_{b_1}$, $a_1:=\alpha_1(1)$ $a_2:=\alpha_2(1)$, and  $a_3:=\alpha_3(1)$, we  obtain
        \begin{align*}
            g=a_1\widehat{\kappa_1}(A)b_1+a_2\left(\widehat{\kappa_2}(A)b_2-\widehat{q_2^1}(A)b_1\right)+a_3\left(\widehat{\kappa_3}(A)b_3-\widehat{q_3^2}(A)b_2-\widehat{q_3^1}(A)b_1\right).
        \end{align*}
        \item ($L\geq 3$): Follows recursively mimicking the previous cases.
    \end{enumerate}
\end{proof}

\begin{remark}\label{remark: span of obs space}
    In Theorem \ref{th: characterization of all g}, once we fix the minimal $A$-conductor polynomials $\kappa_j$, the polynomials $q_j^i$ satisfying \eqref{eq: q polynomials}   are not necessarily unique. Indeed, for any choice of  $q_j^i\in\F[x]$ ($2\leq j\leq L$, $1\leq i\leq j-1$) satisfying \eqref{eq: q polynomials}, we can build the vectors in \eqref{eq: characteristic vectors general case}, i.e.,
    \begin{equation*}\label{eq: generators of g}
       g_j=\widehat{\kappa_j}(A)b_j-\sum_{i=1}^{j-1}\widehat{q_j^i}(A)b_i \qquad \text{ for } 1\leq j\leq L. 
    \end{equation*}  
    Using  formula \eqref{eq: g in general}, the $A$-$\{b_1,\dots, b_L\}$ characteristic space $\mathcal{G}$ (see Definition \ref {def: general admissible vectors}) is spanned by the set $\{g_1,\dots,g_L\}$. Hence, $\mathrm{dim}(\mathcal{G})\leq L$. However, if the spatial observational vectors  $b_1,\dots,b_L$ are linearly independent, then $\mathrm{dim}(\mathcal{G})= L$, as stated in the next corollary.
\end{remark}

\begin{corollary}\label{coro: dim L}
    Let $A:\F^{d}\to\F^d$ be a linear operator and let $b_1,\dots,b_L$ be linearly independent vectors in $\F^d$. Then the $A$-$\{b_1,\dots, b_L\}$ characteristic space $\mathcal{G}$ (see Definition \ref {def: general admissible vectors}) has dimension $L$.
\end{corollary}

\begin{proof}
    From Remark \ref{remark: span of obs space}, we already know that $\mathrm{dim}(\mathcal{G})\leq L$.

    The polynomials $\kappa_j$, for $1\leq j\leq L$, are unique and completely determined (see the statement of Theorem \ref{th: characterization of all g}). However, the polynomials $q_{j}^i$ satisfying \eqref{eq: q polynomials} are not necessarily unique. 

    We will show that once the polynomials $q_{j}^i$ satisfying \eqref{eq: q polynomials} are fixed, the vectors of the form \eqref{eq: characteristic vectors general case} are linearly independent as long as $b_1,\dots,b_L$ are linearly independent. As a result, we will obtain $\mathrm{dim}(\mathcal{G})=L$.

    We will proceed by induction. 

\begin{itemize}
    \item  First case. We recall that $\kappa_1=m_{b_1}$, the minimal $A$-annihilating polynomial of $b_1$. Then, 
    \begin{equation}\label{eq: 1st vector non zero a}
         g_1:=\widehat{\kappa_1}(A)b_1\not=\textbf{0}   
    \end{equation}
    since $\mathrm{deg}(\widehat{\kappa_1})< \mathrm{deg}(\kappa_1)$,  and $\kappa_1$ is the unique monic polynomial with the smallest degree such that $\kappa_1(A)b_1=\textbf{0}$. 
      
    \item Let $1< \ell\leq L$. Assuming that that  $g_1,\dots,g_{\ell-1}$ given by \eqref{eq: characteristic vectors general case}  
    are linearly independent, we will prove that $g_1,\dots, g_{\ell-1},g_\ell$, are linearly independent by separating into two cases.

     If $\kappa_\ell\not\equiv 1$, then $\mathrm{deg}(\widehat{\kappa_\ell})<\mathrm{deg}(\kappa_\ell)$. Thus, the vector $\widehat{\kappa_\ell}(A)b_\ell$ does not belong to $Z(A;b_1,\dots,b_{\ell-1})=\sum_{i=1}^{\ell-1}Z(A;b_i)$ 
   since by definition, $\kappa_\ell$ is the unique monic polynomial of the least degree such that $\kappa_\ell(A)b_\ell\in Z(A;b_1,\dots,b_{\ell-1})$. In particular, $\widehat{\kappa_\ell}(A)b_\ell$ is non-trivial.  Consequently, $g_\ell=\widehat{\kappa_\ell}(A)b_\ell-\sum_{i=1}^{\ell-1}\widehat{q_\ell^i}(A)b_i\not\in Z(A;b_1,\dots,b_{\ell-1})$. Hence, it canxxnot be written as a linear combination of the vectors $g_1=\widehat{\kappa_1}(A)b_1$,..., $g_{\ell-1}=\widehat{\kappa_{\ell-1}}(A)b_{\ell-1}-\sum_{i=1}^{\ell-2}\widehat{q_{\ell-1}^i}(A)b_i$ that lie in $\sum_{i=1}^{\ell-1}Z(A;b_i)$.

    If $\kappa_3\equiv 1$, then $g_\ell=\widehat{\kappa_\ell}(A)b_\ell-\sum_{i=1}^{\ell-1}\widehat{q_\ell^i}(A)b_i=-\sum_{i=1}^{\ell-1}\widehat{q_\ell^i}(A)b_i$. 
    By way of contradiction, suppose that the vectors 
    $g_1$,..., $g_{\ell-1}$, $g_\ell$ are linearly dependent, then there exist scalars $\beta_1,\dots,\beta_\ell$, not all zero, such that 
         \begin{align}\label{eq: dim aux ell}
             \textbf{0}=\beta_1g_1+\dots+\beta_{\ell-1}g_{\ell-1}-\beta_\ell\left(\sum_{i=1}^{\ell-1}\widehat{q_\ell^i}(A)b_i\right).
         \end{align}
    If we apply the operator $A-I$ to the equality above, 
    and note that
    \begin{align*}
        (A-I)g_j&=\cancel{\kappa_j(A)b_j}-\kappa_j(1)b_j-\sum_{i=1}^{j-1}\cancel{q_j^i(A)b_i}+\sum_{i=1}^{j-1}q_j^i(1)b_i=\sum_{i=1}^{j-1}q_j^i(1)b_i-\kappa_j(1)b_j,
    \end{align*}    
    we obtain
         \begin{align}\label{eq: inductive proof}
             \textbf{0}&=-\beta_1\kappa_1(1)b_1+\beta_2\left(q_2^1(A)b_1-\kappa_2(A)b_2\right)+\dots \\
           &\qquad\dots+\beta_{\ell-1}\left(\sum_{i=1}^{\ell-2}q_{\ell-1}^i(1)b_i-\kappa_{\ell-1}(1)b_{\ell-1}\right) +\beta_\ell\left(\sum_{i=1}^{\ell-1}q_{\ell}^i(1)b_i-b_\ell\right).\notag
         \end{align}
         Since by hypothesis $b_1,\dots,b_{\ell-1},b_\ell$ are linearly independent, and the last term in \eqref{eq: inductive proof} is the only one where the vector $b_\ell$ appears, we conclude that $\beta_\ell=0$. Replacing $\beta_\ell=0$ into \eqref{eq: dim aux ell} we get $\textbf{0}=\beta_1g_1+\dots+\beta_{\ell-1}g_{\ell-1}$ 
         with $\beta_1,\dots,\beta_{\ell-1}$ not all zero (since $\beta_\ell$ is already zero), which is a contradiction because of the inductive hypothesis.    
\end{itemize}
    \end{proof}

\begin{proof}[Proof of Theorem \ref{th: general test}]
The proof of Theorem \ref{th: general test} is a direct consequence of Theorem \ref{th: characterization of all g} and Theorem \ref{coro: nec suff cond}.
\end{proof}

    \begin{proof}[Proof of Corollary \ref{coro: at least K general}]
            
        Let $G$ and $B$ as in the statement of Theorem \ref {coro: nec suff cond}. 
        Without assuming any condition on the spectrum of $A$, the result is a direct consequence of Corollary \ref{coro: dim L} as the dimension of the  $A$-$\{b_1,\dots,b_L\}$ characteristic space  $\mathcal{G}$ equals the rank of $B$. Then, as we need the projected subspace $P_W(\mathcal{G})$ to have dimension $K$, we need at least $L=K$ spatial observational vectors $b_j$.
    \end{proof}

    \begin{remark}
As we mention before, the polynomials $\kappa_j$, for $1\leq j\leq L$, are unique and completely determined even though  the polynomials $q_{j}^i$ satisfying \eqref{eq: q polynomials} are not necessarily unique. In what follows we will make particular choices for those polynomials $q_j^i$. These elections will correspond to particular choices of bases for the observational orbits spaces $Z(A;b_1)$, $Z(A;b_1,b_2)$, ..., $Z(A;b_1,\dots,b_L)$.

 For each $1\leq j\leq L$, let $ r_j:=\mathrm{deg} (m_{b_j})=\mathrm{dim}(Z(A;b_j))$ and $s_j:=\mathrm{deg}(\kappa_j).$ 
 
 Notice that for each $2\leq j\leq L$, 
    \begin{equation*}
        \widehat{\kappa_j}=0 \Longleftrightarrow \kappa_j\equiv 1 \Longleftrightarrow b_j\in Z(A;b_1)+\dots +Z(A;b_{j-1}).
    \end{equation*}

        Since $\mathrm{deg}(m_{b_1})=r_1=s_1$, we can consider the basis $\mathcal{B}_1:=\{b_1,Ab_1,\dots,A^{s_1-1}b_1\}$ for $Z(A;b_1)$. Since  $\kappa_2(A)b_2\in Z(A;b_1)$, it can be written as a unique linear combination of the vectors in $\mathcal{B}_1$. Thus, we pick $q_2^1$ as the unique polynomial of degree at most $s_1-1$ such that $\kappa_2(A)b_2=q_2^1(A)b_1$.

    Having chosen $q_2^1(x)$ and a basis $\mathcal{B}_1$ for $Z(A;b_1)$, let us consider a basis $\mathcal{B}_2$ for $Z(A;b_1,b_2)=Z(A;b_1)+Z(A;b_2)$ by extending $\mathcal{B}_1=\{b_1,Ab_1,\dots,A^{s_1-1}b_1\}$ with some vectors from $\{b_2,Ab_2,\dots,A^{r_2-1}b_2\}$ (which is a basis for $Z(A;b_2)$). 
         As $\mathrm{deg}(\kappa_2)=s_2$, we will show that we can extend $\mathcal{B}_1$ to a basis $\mathcal{B}_2$ for $Z(A;b_1,b_2)$
         by choosing some vectors of the form $A^nb_2$ for $n\leq s_2-1$.
         %         the vector of the form $A^nb_2$ from $\{b_2,Ab_2,\dots,A^{r_2-1}b_2\}$ with the \textbf{greatest} $n\geq 0$ that appears in the basis $\mathcal{B}_2$ for $Z(A;b_1,b_2)$ that extends $\mathcal{B}_1$ is $A^{s_2-1}b_2$. 
         %One can show this fact as follows: 
         Indeed, let $\kappa_2(x)=\sum_{j=0}^{s_2}u_jx^j$, with $u_{s_2}=1$. We know that $\kappa_2(A)b_2\in Z(A;b_1)$, thus there exits scalars $v_1,\dots, v_{s_1-1}$ such that 
         $$\sum_{j=0}^{s_2}u_jA^jb_2=\sum_{i=0}^{s_1-1}v_iA^{i}b_1.$$
         In particular,
         \begin{equation}\label{eq: explanation aux}
           A^{s_2}b_2=\sum_{i=0}^{s_1-1}{v_i}A^{i}b_1-\sum_{j=0}^{s_2-1}{u_j}A^jb_2.  
         \end{equation}
         That is, $A^{s_2}b_2$ can be written in terms of 
         $\mathcal{B}_1$ and some vectors of the form $A^nb_2$ with $n< s_2$. Now, given the vector $A^{s_2+1}b_2$, by applying $A$ \ak{to} \eqref{eq: explanation aux} we have
         \begin{equation}\label{eq: explanation aux 2}
             A^{s_2+1}b_2=\sum_{i=0}^{s_1-1}{v_i}A^{i+1}b_1-\sum_{j=1}^{s_2-1}{u_{j-1}}A^jb_2+u_{s_2-1}A^{s_2}b_2.
         \end{equation}
         Since the first $s_1$ vectors on the right-hand side of \eqref{eq: explanation aux 2}  can be written in terms of $\mathcal{B}_1$ (including $A^{s_1+1}b_1$) and since we have already shown that $A^{s_2}b_2$ is a linear combination of vectors of the form $A^nb_2$ with $n< s_2$, we conclude that $A^{s_2+1}b_2$ can also be written as a linear combination of $\mathcal{B}_1$ and some vectors of the form $A^nb_2$ with $n< s_2$. By repeating this argument recursively we obtain that all vectors $A^mb_2$ with $m\geq s_2$  are linear combinations of vectors in $\mathcal{B}_1$ and vectors of the form $A^nb_2$ with $n< s_2$ 
         (in other words, every vector in $Z(A;b_2)$ is a linear combination of vectors in $\mathcal{B}_1$ and vectors $A^nb_2$ with $n< s_2$). Finally, notice that the vector $A^{s_2-1}b_2\not\in Z(A;b_1)$ (i.e., it cannot be written in terms of $\mathcal{B}_1$ since $\kappa_2(x)$ has degree $s_2$ and it is the polynomial with the least degree satisfying $\kappa_2(A)b_2$ that can be written in terms of $\mathcal{B}_1$).
         
        Once we have chosen $\mathcal{B}_2$ in that way, one can pick accordingly unique polynomials $q_3^1$ and $q_{3}^2$ satisfying \eqref{eq: q polynomials} of degree less than $s_1-1$ and $s_2-1$ respectively.

        These choices can be repeated recursively.          
    \end{remark}

\subsection{Proofs for Section 2.2}
In Theorem \ref{negative_conclusion}, we assume that the source term belongs to one-dimensional subspace $W$ of $\F^d$  spanned by a known unit vector $\omega_1$. 
That is, for an arbitrary linear operator $A$ in $\F^d$ with adjoint $A^*$, we  consider the non-homogeneous discrete-time dynamical system
\begin{equation}\label{eq: dyn one c}
    x({n+1}) = A^*x(n) \quad +\underbrace{c_1 \, \omega_1}_{\text{source term } \omega},        
\end{equation}
and spatial measurements at only one location $b$:
\begin{equation}\label{discrete-time-dyn-c}
        %\begin{cases}
         %   x({n+1}) = A^*x(n)+c\, \omega_1, \qquad x(0)=x_0 \\
            y(n) = \langle x(n), b\rangle .  
        %\end{cases}
\end{equation}
To  prove Theorem \ref{negative_conclusion}, we will need several auxiliary results. 

 \begin{remark}\label{remark: solution no 1}
      We recall that if $1$ does not belong to the spectrum of $A$, one strategy to pick the spatial sampling vector is to consider $b:=(I-A)\omega_1$ as in Corollary \ref{coro: at least as many as dim W}. Then, we consider scalars $\{\alpha_n\}_{n=0}^r$ such that $m_b(x)=\sum_{n=0}^r\alpha_n x^n$ is the  minimal $A$-annihilating of $b$. As $m_b$ divides the minimal polynomial of $A$, it holds that $m_b(1)\not=0$. Therefore, by considering $r+1$ space-time measurements $y(0),\dots, y(r)$ and performing the following linear combinations
    \begin{align*}
        \sum_{n=0}^r \alpha_n y(n) &=\langle c_1 \, \omega_1, \sum_{n=0}^r\alpha_n\Lambda_n b\rangle=\langle c_1 \, \omega_1,(I-A)^{-1}m_b(1)b\rangle=c_1 \, m_b(1) \,  \langle \omega_1,\omega_1\rangle, 
    \end{align*}
    we can recover the intensity $c_1$ of the source term $\omega:=c_1 \, \omega_1$ as  $$c_1=\frac{1}{m_b(1)}\sum_{n=0}^r \alpha_n y(n).$$
\end{remark}

The following result provides a simple strategy to pick a space sampling vector $b$ in order to recover the scalar $c_1$ in \eqref{eq: dyn one c} in a wide range of scenarios.
    \begin{proposition}\label{simple_cases}
        Let $A$ be an operator in $\F^d$, and $W$ be the one-dimensional subspace of $\F^d$ spanned by a unit vector $\omega_1\in\F^d$. If there exists an eigenvector $b$ of $A$ such that $\omega_1$ is not orthogonal to $b$, then $b$ is, in fact, complete for Problem \ref{prob: general}. In particular, if $A$ is diagonalizable, it is always possible to recover $c_1$ by using only one sampling vector $b$.
    \end{proposition}

    \begin{proof} 
        Consider an eigenvector $b$  of $A$ such that is not orthogonal to $\omega_1$ (this can be achieved, for example, if $A$ is diagonalizable). If we know two consecutive observations, for instance, $y({0})$ and $y(1)$, then
        $$c_1=\frac{y({1})-\lambda y(0)}{\langle\omega_1,b\rangle},$$
        where $\lambda$ is  such that $Ab=\overline{\lambda} b$. (Notice that in this case, since $b$ is an eigenvector for $A$, its minimal $A$-annihilating polynomial has degree 1.)
    \end{proof}
As a particular case of Theorem \ref{th: characterization of all g} we have the following corollary.

\begin{corollary}
\label{lemma: characterization of b admissible vectors}
Let $A$ be a linear operator in $\F^d$ and let $b\in\F^d$. Let $m_b(x)$ be the minimal $A$-annihilating polynomial of $b$. Then, every =  $A$-$\{b\}$ characteristic vector $g\in\F^d$  (see Definition \ref {def: general admissible vectors}) is a multiple of 
$\widehat{m_b}(A)b$,
where $\widehat{m_b}\in\F[x]$ is given by 
\begin{equation}\label{eq: p in general}
\widehat{m_b}(x):=\frac{m_b(x)-m_b(1)}{x-1}\in\F[x].
\end{equation} 
In particular, if $m_b(1)=0$, then every   non-zero $A$-$\{b\}$ characteristic vector $g$ is an eigenvector of $A$ corresponding to eigenvalue $1$.  
% Moreover, $g$ is a multiple of  
%         $p_b(A)b$ where
%         \begin{equation}\label{eq: g when 1 is root}
%              p_b(x):=\frac{m_b(x)}{x-1}\in\F[x].
%         \end{equation} 
\end{corollary}

As an application of Corollary \ref{lemma: characterization of b admissible vectors}, the following result gives us a general test for discriminating whether a vector $b$ is an appropriate or inappropriate choice as a spatial sampling vector for recovering $c$ in \eqref{eq: dyn one c}.

\begin{lemma}
\label{test}
    Given a linear operator $A$ in $\F^d$, and $W$ a one-dimensional subspace of $\F^d$ spanned by a unit vector $\omega_1\in\F^d$. A spatial sampling vector $b\in\F^d$ is complete for the Problem \ref{prob: general} if and only if 
    \begin{equation*}
         \langle \omega_1, \widehat {m_b}(A)b\rangle\not=0,
    \end{equation*}
    where $m_b$ is the minimal $A$-annihilating polynomial of $b$.    
\end{lemma}

\begin{proof}
    The proof follows from checking Condition \ref{item: 1st cond} in Theorem \ref{coro: nec suff cond} in the case $L=1$, $b_1:=b$, and using Corollary \ref{lemma: characterization of b admissible vectors}.
\end{proof}

\begin{proof}[Proof of Theorem \ref{negative_conclusion}]
    Let $\sigma(A)$ (resp. $\sigma(A^*)$) be the spectrum of $A$ (resp. of $A^*$). First simply notice that $1\in \sigma(A)$ if and only if $1\in \sigma(A^*)$ (we recall that $\sigma(A^*)=\overline{\sigma(A)}$). Therefore, if $1$ is not an eigenvalue for $A$, the space $\ker(I-A^*)^d$ is the trivial subspace $\{{\bf 0}\}$. Thus, Condition \eqref{eq: not observable subspace} is trivially satisfied as the source location $\omega_1$ is assumed different from the zero vector, and there is always a complete set of vectors for Problem \ref{prob: general} by considering only one spatial sampling vector as pointed out in Remark \ref{remark: solution no 1}. 

    Now,  assume $1\in\sigma(A)$. In the notation of Remark \ref{remark: projectors Jordan} given in the Appendix, let us assume $\lambda_1=1$ and so $E_1:=E_{\lambda_1}$ is the projector ($(E_1)^2=E_1$) onto the generalized eigenspace of $A$ corresponding to the eigenvalue $1$.
     Then, from \eqref{eq: ker proj *}, $\omega_1\in\ker(A^*-I)^d$ if and only if $\omega_1\in\ker(I-E_1^*)$.    
    Now, we separate the rest of the proof into two cases. 
    \begin{enumerate}
         \item  Assume $\omega_1\in\ker(I-E_1^*)$, that is, $E_1^*\omega_1=\omega_1$. We claim that there exists a complete $b\in\F^d$ for Problem \ref{prob: general} if and only if $\omega_1\not\in (\ker (A-I))^\perp$. Indeed, since $E_1$ is a polynomial on $A$ (see Remark \ref{remark: projectors Jordan}), we have by Lemma \ref {test} that   $b\in\F^d$ is complete  if and only if          \begin{equation*}
             0\not=\langle \omega_1, \widehat{m_b}(A)b\rangle=\langle E_1^*\omega_1, \widehat{m_b}(A)b\rangle=\langle \omega_1, \widehat{m_b}(A)E_1b\rangle,
         \end{equation*}
        where we recall that $\widehat{m_b}(x)\in\F[x]$ is as in \eqref{eq: p in general}.
        
        Thus, if $b$ is a complete for Problem \ref{prob: general}, it is necessary that
        $E_1b\not=\textbf{0}$. This implies that $m_b(1)=0$ (see Lemma \ref{lemma: appendix} in the Appendix). %In other words, if $m_b(x)$ is the minimal $A$-annihilating polynomial, a         necessary condition is that          $(x-1)$ divides $m_b(x)$ (since we need $b\not \in\ker(E_1)$). 
        Then, by the particular case in Corollary  \ref{lemma: characterization of b admissible vectors}, we have that the vector $\widehat{m_b}(A)b$ is an eigenvector of $A$ corresponding to the eigenvalue $1$. Moreover, notice that
        \begin{equation*}
            \ker(A-I)=\{\widehat {m}_{b}(A)b : \, b\in \F^d \text{ with } m_b(1)=0\}
        \end{equation*} 
        (see Lemma \ref{lemma: appendix}.)
        As a result, we can find a complete $b$ as long as $\omega_1$ is not orthogonal to the eigenspace $\ker(A-I)$. We recall that, if $\omega_1\not\in (\ker (A-I))^\perp$, we are in the setting of Proposition \ref{simple_cases} which provides a strategy for picking an accurate sampling vector $b$.

         \item If $\omega_1\not\in\ker(I-E_1^*)$, we claim that it is always possible to find a complete observational vector $b$  for Problem \ref{prob: general} given by
        \begin{equation}\label{eq: b}
            b:=(A-I)(I-E_1)(I-E_1^*)\omega_1.
        \end{equation}

        First, we will see that $b\not=\textbf{0}$. By assumption,  $(I-E_1^*)\omega_1\not = \textbf{0}$. Then, we have 
        \begin{equation}\label{eq: aux big th}
            0\not=\langle (I-E_1^*)\omega_1,(I-E_1^*)\omega_1 \rangle=\langle \omega_1,(I-E_1)(I-E_1^*)\omega_1 \rangle,
        \end{equation}
        and so $(I-E_1)(I-E_1^*)\omega_1\not=\textbf{0}$. Finally,  we know that
        %$(A-I)(I-E_1)(I-E_1^*)\omega_1\not=\textbf{0}$
        $(I-E_1)$ projects onto $\bigoplus_{\lambda\in \sigma (A), \lambda\not=1}\ker(A-\lambda I)^d$ and has kernel $\ker(A-I)^d$ (see Remark \ref{remark: projectors Jordan}).  Since $\ker (A-I)\subset\ker(A-I)^d$, we get that  $b=(A-I)(I-E_1)(I-E_1^*)\omega_1\not=\textbf{0}$.
        Moreover, as $$(I-E_1)(I-E_1^*)\omega_1\in \bigoplus_{\lambda\in \sigma (A), \lambda\not=1}\ker(A-\lambda I)^d,$$ and that the direct sum of generalized eigenspaces of $A$ is invariant under $A$ (and hence under any polynomial on $A$), we have that  $b\in \bigoplus_{\lambda\in \sigma (A), \lambda\not=1}\ker(A-\lambda I)^d$. Hence,  $m_b(1)\not =0$. 

        Now, consider $\widehat{m_b}(x)=\frac{m_b(x)-m_b(1)}{x-1}\in \F[x].$
        Then, we have that
        
       \begin{align*}
            \widehat{m_b}(A)b&=\widehat{m_b}(A)(A-I)(I-E_1)(I-E_1^*)\omega_1\\
                %&=(A-I)\widehat{m_b}(A)(I-E_1)(I-E_1^*)\omega_1\\
                &=\big(m_b(A)-m_b(1)\big)(I-E_1)(I-E_1^*)\omega_1.
        \end{align*}
        We will show that $$m_b(A)(I-E_1)(I-E_1^*)\omega_1={\bf 0}$$
        by using the fact that $I-E_1=(I-E_1)^2$ and that $I-E_1=p(A)$, where $p(x):=\widetilde{h_{1}}(x)(x- 1)^{n_1}$ is the polynomial given in \eqref{eq: E and I-E} (for eigenvalue $\lambda_i=1$). Notice that $(x-1)$ divides $p(x)$, and so we can factor $p(x)=\frac{p(x)}{(x-1)}(x-1)=\widehat{p}(x)(x-1)$. Then, when evaluating at $A$ we have
        \begin{align*}
            m_b(A)(I-E_1)(I-E_1^*)\omega_1&=m_b(A)(I-E_1)(I-E_1)(I-E_1^*)\omega_1\\
            &=m_b(A)\widehat{p}(A)(A-I)(I-E_1)(I-E_1^*)\omega_1\\
            &=\widehat{p}(A)m_b(A)b=\textbf{0}.
        \end{align*}         
        %  Indeed, this follows by recalling that $(I-E_1)^2=I-E_1$ and that
        %  $I-E_1=h_1(A)(A-I)^{n_1}$ 
        %  where $n_1$ is the multiplicity of $\lambda_1=1$ in the minimal polynomial of $A$ and $h_1(x)$ is given by \eqref{eq: pol I-E}.
        %  \begin{align*}
        % % p_b(x)(h(x)(x-1)^{n_1})^2(x-1)=
        %  &h(x)(x-1)^{n_1-1}m_b(x)h(x)(x-1)^{n_1}(x-1)-m_b(1)(h(x)(1-x)^{n_1})^2
        %  \end{align*}
        % % and so
        % \begin{align*}
        %      p_b(A)b&=
        % \end{align*}
        % %we are in the setting of Proposition \ref{simple_cases} since $\omega_1$ should be such that $\langle\omega_1,b\rangle\not=0$ for some eigenvector $b$ of $A$ corresponding to eigenvalue $\lambda\not=1$. Then, as in the proof of \ref{simple_cases}, we can select that vector $b$ as a particular solution of the one intensity -- one sampler problem (for which $c$ can be recovered after two space-time measurements).
       Therefore,
        \begin{equation}\label{eq: pbA}
           \widehat{m_b}(A)b=-m_b(1)(I-E_1)(I-E_1^*)\omega_1. 
        \end{equation}
        Using \eqref{eq: pbA} and \eqref{eq: aux big th} we obtain
        \begin{align*}
            \langle \omega_1,\widehat{m_b}(A)b\rangle &=-m_b(1)\langle \omega_1, (I-E_1)(I-E_1^*)\omega_1\rangle\\&=|m_b(1)|\|(I-E_1^*)\omega_1\|^2\not =0.
        \end{align*}
        Thus, by  Lemma \ref{test} the vector $b$ defined by \eqref{eq: b} is complete for Problem \ref{prob: general}.
    \end{enumerate}

\end{proof}

\section{Appendix: Matrix Analysis}

    \begin{theorem}\label{th: interpol th}  \cite[Theorem 13.2]{higham2008functions}
        Let $f$ be an analytic function defined on %the spectrum of a matrix $A\in\C^{d\times d}$ and let 
        a domain containing the $s$ roots of the polynomial
        $m_b(x)=\prod_{i=1}^s(x-\lambda_i)^{n_i}$ which is the minimal $A$-annihilating of $b$ (each root $\lambda_i$ has multiplicity $n_i$). Then $$f(A)b = q(A)b$$ for $q$ the (unique)
        Hermite interpolating polynomial of degree less than or equal to  $\sum_{i=1}^s n_i=deg(m_b)$ that satisfies
        the interpolation conditions
        \begin{equation}\label{eq: interpolating pol}
          \frac{d^j}{dx^j}q(\lambda_i)=\frac{d^j}{dx^j}f(\lambda_i) \qquad \text{ for all } j=0,\dots, n_i-1, \, i=1,\dots, s  .
        \end{equation}
        In particular, if $f$ is an analytic function on the spectrum of $A$ and $m_b$ is replaced by the minimal polynomial of $A$, then $f(A)=q(A)$. 
    \end{theorem}

\begin{remark}\label{remark: projectors Jordan}
    Let $A$ be a linear operator in $\F^d$, and let $$m_A(x):=\prod_{j=1}^s(x-\lambda_j)^{n_j}$$ be its minimal polynomial (i.e., the unique monic polynomial with the least degree satisfying $m_A(A)={\bf 0}$). (Notice that, $\lambda_1, \ldots, \lambda_s\in\C$ are the distinct eigenvalues of $A$ and, when writing $A$ in Jordan form, $n_i$ determines the size of the largest Jordan block in which $\lambda_i$ appears.) 
    By the Lagrange-Hermite formula, the Hermite interpolating polynomial satisfying \eqref{eq: interpolating pol} for the case $f(x)=1$ is exactly the constant polynomial $1$ and it can be written as
        \begin{equation}\label{eq: hermite}
            1=\sum_{i=1}^s\left[\left(\sum_{k=0}^{n_i-1} \frac{1}{k !} \phi_i^{(k)}\left(\lambda_i\right)\left(x-\lambda_i\right)^k\right) \prod_{j \neq i}\left(x-\lambda_j\right)^{n_j}\right]
        \end{equation}
        where $ \phi_i(x):=1 / \prod_{j \neq i}\left(x-\lambda_j\right)^{n_j}$ (see \cite{higham2008functions}). 
        Let
        \begin{equation}\label{eq: p_lambda}
            p_{\lambda_i}(x):= h_{\lambda_i}(x) \prod_{j \neq i}\left(x-\lambda_j\right)^{n_j}, \quad \text{ where } \quad h_{\lambda_i}(x):=\sum_{k=0}^{n_i-1} \frac{1}{k !} \phi_i^{(k)}\left(\lambda_i\right)\left(x-\lambda_i\right)^k\,
        \end{equation}
        and consider the operators 
        \begin{equation}
            E_{\lambda_i}:=p_{\lambda_i}(A).
        \end{equation}        
        Therefore,
        \begin{equation}\label{eq: partition of the identity}
            I= E_{\lambda_1}+\dots+ E_{\lambda_s}.
            %\sum_{i=1}^s\underbrace{\left[\left(\sum_{j=0}^{n_i-1} \frac{1}{j !} \phi_i^{(j)}\left(\lambda_i\right)\left(A-\lambda_iI\right)^j\right) \prod_{j \neq i}\left(A-\lambda_jI\right)^{n_j}\right]}_{E_{\lambda_i}}        
        \end{equation}
        It is easy to check that
        \begin{equation*}
            E_{\lambda_i}E_{\lambda_k}=\begin{cases}
                {\bf 0} & \text{ if } i\not=k\\
                E_{\lambda_i} & \text{ if } i=k
            \end{cases}.
        \end{equation*}
        Indeed, if $i\not=k$, $p_{\lambda_i}(x)p_{\lambda_k}(x)$ has the same roots as $m_A(x)$, and so $m_A|p_{\lambda_i}p_{\lambda_k}$ having that $E_{\lambda_i}E_{\lambda_k}=p_{\lambda_i}(A)p_{\lambda_k}(A)={\bf 0}$. Also, notice that when we multiply by $E_{\lambda_i}$ on both sides of Identity \eqref{eq: partition of the identity}  we get that $E_{\lambda_i}=(E_{\lambda_i})^2$ (i.e., $E_{\lambda_i}$ is an idempotent operator). %Thus, for each $1\leq i\leq s$, the operator $E_{\lambda_i}$ is a idempotent projector. 
        Moreover, $E_{\lambda_i}$ and $I-E_{\lambda_i}$ (where $I-E_{\lambda_i}=\sum_{j\not= i}E_{\lambda_j}$) are idempotent projectors such that $I=E_{\lambda_i}+ (I-E_{\lambda_i})$ and $E_{\lambda_i}(I-E_{\lambda_i})={\bf 0}$. In addition, we can write 
        \begin{equation}\label{eq: E and I-E}
            E_{\lambda_i}=h_{\lambda_i}(A) \prod_{j \neq i}\left(A-\lambda_jI\right)^{n_j}, \qquad \text{ and } \qquad I- E_{\lambda_i} = \widetilde{h_{\lambda_i}}(A)(A-\lambda_i I)^{n_i}  
        \end{equation}
        where the polynomial $h_{\lambda_i}(x)$ is given in \eqref{eq: p_lambda}, and the explicit formulas for $\widetilde{h_{\lambda_i}}(x)$ can be deduced from \eqref{eq: hermite} obtaining
         \begin{equation}\label{eq: pol I-E}
                \widetilde{h_{\lambda_i}}(x):=\sum_{k=1, k\not=j}^s h_{\lambda_k}(x) \prod_{j \neq k, j\not= i}\left(x-\lambda_j I\right)^{n_j}.
            \end{equation}
        Therefore, using the above characterization and the decomposition of $\F^d$ as a direct sum  (not necessarily orthogonal) of the generalized eigenspaces of $A$
        $$\F^d={\bigoplus_{j=1}^{s}} \ker(A-\lambda_j I)^d,$$        
        we obtain
        \begin{equation}\label{eq: ker proj}
             \ker(E_{\lambda_i})={\bigoplus_{j\not=i}} \ker(A-\lambda_j I)^d,  \qquad  
         \ker(I-E_{\lambda_i})= \ker(A-\lambda_i I)^d.  
        \end{equation}
        In sum, $E_{\lambda_i}$ is the projection onto the generalized eigenspace corresponding to eigenvalue $\lambda_i$ whose kernel is the direct sum of the rest of the generalized eigenspaces of $A$. 
        Finally, taking the adjoint in the resolution of the identity \eqref{eq: partition of the identity} we get
        \begin{align*}
            I&=\left(\sum_{i=1}^s\left[\left(\sum_{j=0}^{n_i-1} \frac{1}{j !} \phi_i^{(j)}\left(\lambda_i\right)\left(A-\lambda_iI\right)^j\right) \prod_{j \neq i}\left(A-\lambda_jI\right)^{n_j}\right]
            \right)^*\\           &=\sum_{i=1}^s\left[\left(\sum_{j=0}^{n_i-1} \frac{1}{j !} \overline{\phi_i^{(j)}\left(\lambda_i\right)}\left(A^*-\overline{\lambda_i}I\right)^j\right) \prod_{j \neq i}\left(A^*-\overline{\lambda_j}I\right)^{n_j}\right]\\           &=\sum_{i=1}^s\underbrace{\left[\left(\sum_{j=0}^{n_i-1} \frac{1}{j !} (\phi_i^*)^{(j)}\left(\overline{\lambda_i}\right)\left(A^*-\overline{\lambda_i}I\right)^j\right) \prod_{j \neq i}\left(A^*-\overline{\lambda_j}I\right)^{n_j}\right]}_{E_{\lambda_i}^*}
        \end{align*}
        where $ \phi_i^*(x):=1 / \prod_{j \neq i}\left(x-\overline{\lambda_j}\right)^{n_j}$, which means that the operators $\{E_{\lambda_i}^*\}_{i=1}^s$ are the corresponding idempotent projectors onto the generalized eigenspaces associated to $A^*$. Using the notation given in \eqref{eq: p_lambda}, we can write $E_{\lambda_i}^*=p_{\overline{\lambda}_i}(A^*)$. Therefore, in analogy to \eqref{eq: ker proj}  
        \begin{equation}\label{eq: ker proj *}
       \begin{array}{cc}    \ker(E_{\lambda_i}^*)=\displaystyle{\oplus_{ j\not=i}} \ker( A^*-\overline{\lambda_j} I)^d, &\qquad \ker(I-E_{\lambda_i}^*)= \ker( A^*-\overline{\lambda_i} I)^d.
       \end{array} 
    \end{equation}
    We remark that it is not necessarily true that the projectors $E_{\lambda_i}$ are self-adjoint operators, in fact, they are not orthogonal projections in general. 
    In the special case when $A$ is in Jordan form (i.e., the matrix representation of the operator $A$ with respect to the canonical basis gives rise to its Jordan matrix representation), we have that the generalized eigenspaces are mutually orthogonal, and so the projectors  $E_{\lambda_i}$ are orthogonal projections. In this  case, $\ker(A-\lambda_i I) = \ker(A^*-\overline{\lambda_i} I)$.       
\end{remark}

\begin{lemma}\label{lemma: appendix}
Let $A$ be a linear operator in $\F^d$, and $b$ be a vector in $\F^d$. Let $\lambda_i$ be an eigenvalue of $A$ and $E_{\lambda_i}$ be the projection operator onto the generalized eigenspace of $A$ corresponding to $\lambda_i$. Then, $E_{\lambda_i}b\not=\textbf{0}$ if and only if $m_b(\lambda_i)=0$ (where $m_b$ is the unique minimal $A$-annihilating polynomial of $b$.)

Moreover, the eigenspace corresponding to $\lambda_i$ can be characterized by
\begin{equation*}
   %  \ker(A-\lambda_i I)=\left\{\widehat m_b(A)b : \, b\in \F^d \text{ with } m_b(\lambda_i)=0\right\}.
            \ker(A-\lambda_i I)=\left\{p_{b,\lambda_i}(A)b : \, b\in \F^d \text{ with } m_b(\lambda_i)=0, \text{ and } p_{b,\lambda_i}(x):=\frac{m_b(x)}{x-\lambda_i}\right\}.
        \end{equation*}
(Notice that when $\lambda_i=1$, then $p_{b,1}(x)=\widehat{m_b}(x)$.)        
\end{lemma}

\begin{proof}

First, observe that $E_{\lambda_i}b\ne\textbf{0}$ if and only if $p_{\lambda_i}(A)b\not=\textbf{0}$ where $p_{\lambda_i}(x)$ is given by \eqref{eq: p_lambda}. From its expression $p_{\lambda_i}(\lambda_i)\ne 0$, and $p_{\lambda_i}(\lambda_j)= 0$ for all $j\ne i$. 
Thus, if $(x-\lambda_i)$ is not a factor of the minimal $A$-annihilator $m_b$ of $b$, then, $p_{\lambda_i}(A)b=\textbf{0} $ since for this case $m_b$ divides $p_{\lambda_i}$. 
Conversely, if $m_b(\lambda_i)=0$ then $(x-\lambda_i)$ is a factor of all $A$-annihilators of $b$. Since $p_{\lambda_i}(\lambda_i)\ne 0$, then $p_{\lambda_i}(A)b\not=\textbf{0}$. 
    
    For the last assertion, it is easy to check that
 \begin{equation*}
   % \left\{\widehat m_b(A)b : \, b\in \F^d \text{ with } m_b(\lambda_i)=0\right\}\subseteq \ker(A-\lambda_i I).
            \left\{p_{b,\lambda_i}(A)b : \, b\in \F^d \text{ with } m_b(\lambda_i)=0, \text{ and } p_{b,\lambda_i}(x):=\frac{m_b(x)}{x-\lambda_i}\right\}\subseteq \ker(A-\lambda_i I).
    \end{equation*}
    Indeed, 
    $$(A-\lambda_i I)p_{b,\lambda_i}(A)b=m_b(A)b=\textbf{0}.$$
    %and so $p_{b,\lambda_i}(A)b$ turns out to be an eigenvector of $A$ corresponding to eigenvalue $\lambda_i$ \ak{ (not necessarily:it is possible that $p_{b,\lambda_i}(A)b=0$. However, the inclusion above is still true )}. 
    The other inclusion also follows easily by noticing that, in general, a vector $v$ is an eigenvector of $A$ corresponding to eigenvalue $\lambda_i$ if and only if its minimal $A$-annihilating polynomial is $m_v(x)=(x-\lambda_i)$. Therefore, if $v\in \F^d$ is an eigenvector of $A$ corresponding to eigenvalue $\lambda_i$, we have  that 
    % \ak{$$\widehat m_v(x)=\frac{m_v(x)}{x-\lambda_i}=\frac{x-\lambda_i}{x-\lambda_i}\equiv 1.$$} 
$$p_{b,\lambda_i}(x)=\frac{m_v(x)}{x-\lambda_i}=\frac{x-\lambda_i}{x-\lambda_i}\equiv 1.$$
    So, we can trivially write $v=p_{v,\lambda_i}(A)v$ which belongs to
  \begin{equation*}
  %   \left\{\widehat m_b(A)b : \, b\in \F^d \text{ with } m_b(\lambda_i)=0\right\}.
           \left\{p_{b,\lambda_i}(A)b : \, b\in \F^d \text{ with } m_b(\lambda_i)=0, \text{ and } p_{b,\lambda_i}(x):=\frac{m_b(x)}{x-\lambda_i}\right\}.
    \end{equation*}
\end{proof}

\bibliographystyle{siam}
\bibliography{Akram_refs}

\begin{thebibliography}{10}

\bibitem{AKh17}
{\sc R.~Aceska and Y.~H. Kim}, {\em Scalability of frames generated by
  dynamical operators}, Frontiers in Applied Mathematics and Statistics, 3
  (2017), p.~22.

\bibitem{APT15}
{\sc R.~Aceska, A.~Petrosyan, and S.~Tang}, {\em Multidimensional signal
  recovery in discrete evolution systems via spatiotemporal trade off}, Sampl.
  Theory Signal Image Process., 14 (2015), pp.~153--169.

\bibitem{ACCP21}
{\sc A.~Aguilera, C.~Cabrelli, D.~Carbajal, and V.~Paternostro}, {\em Dynamical
  sampling for shift-preserving operators}, Appl. Comput. Harmon. Anal., 51
  (2021), pp.~258--274.

\bibitem{ACCMP17}
{\sc A.~Aldroubi, C.~Cabrelli, A.~F. \c{C}akmak, U.~Molter, and A.~Petrosyan},
  {\em Iterative actions of normal operators}, J. Funct. Anal., 272 (2017),
  pp.~1121--1146.

\bibitem{ACMT17}
{\sc A.~Aldroubi, C.~Cabrelli, U.~Molter, and S.~Tang}, {\em Dynamical
  sampling}, Applied and Computational Harmonic Analysis, 42 (2017),
  pp.~378--401.
\newblock doi: 10.1016/j.acha.2015.08.014.

\bibitem{ADK13}
{\sc A.~Aldroubi, J.~Davis, and I.~Krishtal}, {\em Dynamical sampling:
  time-space trade-off}, Appl. Comput. Harmon. Anal., 34 (2013), pp.~495--503.

\bibitem{ADK15}
{\sc A.~Aldroubi, J.~Davis, and I.~Krishtal}, {\em Exact reconstruction of
  signals in evolutionary systems via spatiotemporal trade-off}, Journal of
  Fourier Analysis and Applications, 21 (2015), pp.~11--31.

\bibitem{AGK23}
{\sc A.~Aldroubi, L.~Gong, and I.~Krishtal}, {\em Recovery of rapidly decaying
  source terms from dynamical samples in evolution equations}, Sampling Theory,
  Signal Processing, and Data Analysis, 21 (2023), p.~15.

\bibitem{AGHJKR21}
{\sc A.~Aldroubi, K.~Gr{\"o}chenig, L.~Huang, P.~Jaming, I.~Krishtal, and J.~L.
  Romero}, {\em Sampling the flow of a bandlimited function}, The Journal of
  Geometric Analysis, 31 (2021), pp.~9241--9275.

\bibitem{AHKK23}
{\sc A.~Aldroubi, L.~Huang, K.~Kornelson, and I.~Krishtal}, {\em Predictive
  algorithms in dynamical sampling for burst-like forcing terms}, Applied and
  Computational Harmonic Analysis, 65 (2023), pp.~322--347.

\bibitem{AHKLLV18}
{\sc A.~Aldroubi, L.~Huang, I.~Krishtal, A.~Ledeczi, R.~R. Lederman, and
  P.~Volgyesi}, {\em Dynamical sampling with additive random noise}, Sampl.
  Theory Signal Image Process., 17 (2018), pp.~153--182.

\bibitem{AK16}
{\sc A.~Aldroubi and I.~Krishtal}, {\em Krylov subspace methods in dynamical
  sampling}, Sampl. Theory Signal Image Process., 15 (2016), pp.~9--20.

\bibitem{AP17}
{\sc A.~Aldroubi and A.~Petrosyan}, {\em Dynamical sampling and systems from
  iterative actions of operators}, in Frames and other bases in abstract and
  function spaces, Appl. Numer. Harmon. Anal., Birkh\"auser/Springer, Cham,
  2017, pp.~15--26.

\bibitem{RBD21}
{\sc R.~Alexandru, T.~Blu, and P.~L. Dragotti}, {\em Diffusion {SLAM}:
  localizing diffusion sources from samples taken by location-unaware mobile
  sensors}, IEEE Trans. Signal Process., 69 (2021), pp.~5539--5554.

\bibitem{AD20}
{\sc R.~Alexandru and P.~L. Dragotti}, {\em Reconstructing classes of
  non-bandlimited signals from time encoded information}, IEEE Trans. Signal
  Process., 68 (2020), pp.~747--763.

\bibitem{BH23}
{\sc R.~Beinert and M.~Hasannasab}, {\em Phase retrieval and system
  identification in dynamical sampling via {P}rony's method}, Adv. Comput.
  Math., 49 (2023), p.~Paper No. 56.

\bibitem{BK23}
{\sc F.~Bozkurt and K.~Kornelson}, {\em Norm retrieval from few spatio-temporal
  samples}, J. Math. Anal. Appl., 519 (2023), pp.~Paper No. 126804, 17.

\bibitem{CMPP20}
{\sc C.~Cabrelli, U.~Molter, V.~Paternostro, and F.~Philipp}, {\em Dynamical
  sampling on finite index sets}, J. Anal. Math., 140 (2020), pp.~637--667.

\bibitem{chaparro}
{\sc M.~Chaparro, C.~Gogorza, A.~Lavat, S.~Pazos, and A.~Sinito}, {\em
  Preliminary results of magnetic characterisation of different soils in tandil
  region (argentina) affected by the pollution of metallurgical factory}, Eur.
  J. Environ. Eng. Geophys, 7 (2002), p.~35.

\bibitem{CJS15}
{\sc C.~Cheng, Y.~Jiang, and Q.~Sun}, {\em Spatially distributed sampling and
  reconstruction}, CoRR, abs/1511.08541 (2015).

\bibitem{CH19}
{\sc O.~Christensen and M.~Hasannasab}, {\em Frame properties of systems
  arising via iterated actions of operators}, Appl. Comput. Harmon. Anal., 46
  (2019), pp.~664--673.

\bibitem{CH23}
\leavevmode\vrule height 2pt depth -1.6pt width 23pt, {\em Frames and
  generalized operator orbits}, Sampling Theory, Signal Processing, and Data
  Analysis, 21 (2023), p.~22.

\bibitem{DMM21}
{\sc R.~D\'{\i}az~Mart\'{\i}n, I.~Medri, and U.~Molter}, {\em Continuous and
  discrete dynamical sampling}, J. Math. Anal. Appl., 499 (2021), pp.~Paper No.
  125060, 19.

\bibitem{FS19}
{\sc D.~Freeman and D.~Speegle}, {\em The discretization problem for continuous
  frames}, Adv. Math., 345 (2019), pp.~784--813.

\bibitem{GRUV15}
{\sc K.~Gr\"{o}chenig, J.~L. Romero, J.~Unnikrishnan, and M.~Vetterli}, {\em On
  minimal trajectories for mobile sampling of bandlimited fields}, Appl.
  Comput. Harmon. Anal., 39 (2015), pp.~487--510.

\bibitem{higham2008functions}
{\sc N.~J. Higham}, {\em Functions of matrices: theory and computation}, SIAM,
  2008.

\bibitem{Hoffman}
{\sc K.~Hoffman and R.~Kunze}, {\em Linear algebra}, Prentice-Hall, Inc.,
  Englewood Cliffs, N.J., second~ed., 1971.

\bibitem{KS19}
{\sc Z.~A. Kasumov and A.~S. Shukurov}, {\em On frame properties of iterates of
  a multiplication operator}, Results Math., 74 (2019), pp.~Paper No. 84, 8.

\bibitem{MMM21}
{\sc R.~D. Mart\'{\i}n, I.~Medri, and U.~Molter}, {\em Dynamical sampling: a
  view from control theory}, in Excursions in harmonic analysis. {V}ol. 6,
  Appl. Numer. Harmon. Anal., Birkh\"{a}user/Springer, Cham, [2021] \copyright
  2021, pp.~269--295.

\bibitem{Men22}
{\sc T.~Mengestie}, {\em Closed range weighted composition operators and
  dynamical sampling}, J. Math. Anal. Appl., 515 (2022), pp.~Paper No. 126387,
  11.

\bibitem{MT23}
\leavevmode\vrule height 2pt depth -1.6pt width 23pt, {\em Closed range
  {V}olterra-type integral operators and dynamical sampling}, Monatsh. Math.,
  202 (2023), pp.~161--170.

\bibitem{MBD15}
{\sc J.~Murray-Bruce and P.~L. Dragotti}, {\em Estimating localized sources of
  diffusion fields using spatiotemporal sensor measurements}, IEEE Transactions
  on Signal Processing, 63 (2015), pp.~3018--3031.

\bibitem{MD17}
{\sc J.~Murray-Bruce and P.~L. Dragotti}, {\em A sampling framework for solving
  physics-driven inverse source problems}, IEEE Trans. Signal Process., 65
  (2017), pp.~6365--6380.

\bibitem{RDCV12}
{\sc J.~Ranieri, I.~Dokmani{\'c}, A.~Chebira, and M.~Vetterli}, {\em Sampling
  and reconstruction of time-varying atmospheric emissions}, in 2012 IEEE
  International Conference on Acoustics, Speech and Signal Processing (ICASSP),
  2012, pp.~3673--3676.

\bibitem{Tan17}
{\sc S.~Tang}, {\em System identification in dynamical sampling}, Adv. Comput.
  Math., 43 (2017), pp.~555--580.

\bibitem{UZ21}
{\sc A.~Ulanovskii and I.~Zlotnikov}, {\em Reconstruction of bandlimited
  functions from space-time samples}, J. Funct. Anal., 280 (2021), pp.~Paper
  No. 108962, 14.

\bibitem{ZLL17}
{\sc Q.~Zhang, B.~Liu, and R.~Li}, {\em Dynamical sampling in multiply
  generated shift-invariant spaces}, Applicable Analysis, 96 (2017),
  pp.~760--770.

\end{thebibliography}
\end{document}